\documentclass[11pt,a4paper]{article}

\usepackage[margin=1in]{geometry}
\usepackage{amsmath,amsthm,amssymb,mathtools}
\usepackage[numbers,sort&compress]{natbib}
\usepackage[hidelinks]{hyperref}
\usepackage[T1]{fontenc}
\usepackage{lmodern}
\usepackage{microtype}

\theoremstyle{plain}
\newtheorem{theorem}{Theorem}
\newtheorem{lemma}{Lemma}
\theoremstyle{definition}
\newtheorem{example}{Example}
\newtheorem{definition}{Definition}
\newtheorem{remark}{Remark}
\newtheorem{corollary}{Corollary}

\newcommand{\E}{\mathbb{E}}

\newcommand{\R}{\mathbb{R}}

\title{Neighbour--count dependent thinning of Poisson point processes:\\
correlation structure and Poisson approximation}

\author{Kateryna Hlyniana\\[0.5em]
\small School of Mathematics, Jilin University, Changchun, China\\
\small Institute of Mathematics, National Academy of Sciences of Ukraine, Kyiv, Ukraine\\
\small \texttt{glinkate@gmail.com}\\
\small ORCID: 0000-0002-2463-9701}
\date{}

\begin{document}
\maketitle

\begin{abstract}
We study a finite-range dependent thinning of a locally finite point process: given a radius $r>0$
and a retention rule $p:\mathbb N_0\to[0,1]$, each point is retained with probability
$p(n_r)$, where $n_r$ is the number of other points within distance $r$.
For an inhomogeneous Poisson point process and Cox input, we compute the exact intensity of the thinned process
via a Poisson-mixture formula.
For homogeneous Poisson input, we obtain an explicit representation of the pair correlation function,
showing that it equals $1$ beyond distance $2r$ and giving contact-scale asymptotics that distinguish
inhibitory and count-favouring regimes through the values $p(0),p(1),p(2)$.
Finally, for the restriction to a compact window we give quantitative Poisson process approximations
(with matched intensity) via coupling in total variation, Laplace functional comparison, and Stein's method.
\end{abstract}


\bigskip
\section{Introduction}

Thinning is one of the most widely used transformations in spatial point process theory: starting from a tractable input (often Poisson), one deletes points according to a local rule in order to create inhibition or clustering while preserving basic geometric symmetries. Classical examples are the Mat\'ern type~I and type~II models, which implement hard-core behaviour by keeping only isolated points or local ``winners'' of i.i.d.\ marks; see, e.g., \cite{BaddeleyRubakTurner2015,ChiuEtAl2013,LastPenrose2017}. Beyond their direct modelling role, thinnings are also a useful device for approximation: in regimes where the interaction range is small (or the configuration is dilute), a dependent thinning may behave close to an independent thinning, and hence close to a Poisson process, which simplifies the analysis of counts and of smooth functionals.

This paper analyses a particularly simple dependent thinning that bases its decision on a \emph{neighbour count}. Fix $r>0$ and a measurable rule $p:\mathbb N_0\to[0,1]$. Given a locally finite configuration $X$ and a point $x\in X$, let
\[
n_r(x;X):=X\bigl(B(x,r)\setminus\{x\}\bigr)
\]
be the number of other points within radius $r$ of $x$. The transformation $T_r$ retains $x$ with probability $p(n_r(x;X))$, using independent $\mathrm{Uniform}(0,1)$ marks across points (Definition~\ref{def:Ta}). The rule is $r$-local and equivariant under Euclidean isometries (Remark~\ref{rem:equivariance}). It contains Mat\'ern type~I as the special case $p(n)=\mathbf 1\{n=0\}$ (Example~\ref{ex:maternI}), and it also covers soft-core families such as $p(n)=q\,s^n$ (Example~\ref{ex:geom-softcore}). Importantly, unlike Mat\'ern-type rules that are primarily inhibitory, the same framework naturally allows \emph{count-favouring} choices with $p(1)>p(0)$; such rules can create attraction already at contact scales (Example~\ref{ex:count-favouring}). A key structural feature is that all dependence introduced by $T_r$ is confined to distances $\le 2r$: if two regions are separated by more than $2r$, the restrictions of $T_r(X)$ are independent (Lemma~\ref{lem:finite-range}). This finite-range property drives both our explicit second-order calculations and our Poisson approximation bounds.

Our contributions are twofold: (A) explicit first- and second-order structure for $T_r$ under Poisson input (including contact-scale asymptotics that separate attraction from inhibition), and (B) quantitative Poisson approximation on compact windows via three complementary methods, each natural for a different metric.

{(A) Exact intensity and explicit pair correlation.}
For an inhomogeneous Poisson input $X\sim \mathrm{PPP}(\lambda )$, we derive an exact intensity formula for $T_r(X)$ in the form of a Poisson mixture (Lemma~\ref{lem:intensity-Poisson-Cox}); the same argument extends to Cox input by conditioning on the directing measure (Lemma~\ref{lem:intensity-Poisson-Cox}). For homogeneous Poisson input $X\sim\mathrm{PPP}(\lambda)$ we obtain a closed-form representation of the pair–correlation function $g_{Y_r}$ (Theorem~\ref{thm:g-exact-PPP}): it reduces to an expectation involving three independent Poisson counts associated with the standard two-ball overlap diagram. This yields, in particular, the exact finite-range property $g_{Y_r}(h)=1$ for $\|h\|>2r$. We then identify the leading contact-scale behaviour as $r\downarrow 0$ (Corollaries~\ref{lem:g-first-order} and \ref{lem:g-first-order-p0zero}). These expansions make transparent how the values $p(0),p(1),p(2)$ control repulsion versus attraction: inhibitory rules (e.g.\ Mat\'ern I or soft-core with $p(1)<p(0)$) exhibit $g<1$ at contact and a classical ``ring'' overshoot on $(r,2r)$, while count-favouring rules with $p(1)>p(0)$ can produce $g>1$ already for $\|h\|\le r$ (Section~\ref{subsec:examples-attract-repel}).

{(B) Poisson approximation in three metrics.}
Let $W\subset\R^d$ be a compact subset and write $Y_{r,W}$ for the restriction of $T_r(X)$ on  $W$. Even for Poisson input, $Y_{r,W}$ is typically not Poisson because the retention probabilities depend on the local configuration. We quantify the distance between $\mathcal L(Y_{r,W})$ and a Poisson process on $W$ with matched intensity $\lambda'=\lambda\,\E[p(\mathrm{Poisson}(\lambda v_d r^d))]$ by three routes:
\begin{itemize}
\item A \emph{direct coupling} of $T_r(X)$ to an independent thinning of the same input $X$ yields a clean total-variation bound on $\mathbf N(W)$ (Theorem~\ref{thm:coupling-TV}), with an explicit small-$r$ Lipschitz regime (Corollary~\ref{cor:small-mu-TV}); compare with standard coupling principles for point processes \cite{Lindvall1992}.
\item A \emph{Laplace-functional} comparison gives a bound on the log-Laplace discrepancy of $Y_{r,W}$ and the matched Poisson process, with an error controlled by the interaction range $2r$ (Theorem~\ref{thm:laplace-control}). This route is useful for controlling distributional errors for smooth test functionals.
\item A \emph{Stein bound} in the Barbour--Brown $d_2$ metric \cite{BarbourBrown1992} (in the form developed for point processes in \cite{Schuhmacher2009,Schuhmacher2009Bernoulli}) expresses the approximation error through the integrated short-range deviation of the pair correlation,
\[
\int_{\|h\|\le 2r} |g_{Y_r}(h)-1|\,dh
\qquad\text{(Theorem~\ref{thm:stein-main}).}
\]
Combining this with our contact-scale expansions yields explicit small-$r$ rates governed by $p(0),p(1),p(2)$ (Theorem~\ref{prop:stein-smallr}). For moderate $r$ we also provide a non-asymptotic bound under a discrete Lipschitz condition on $p$, using a Poisson Poincar\'e inequality \cite{LastPenrose2011PTRF} to control the same integral (Theorem~\ref{thm:moderate-r}).
\end{itemize}
Taken together, these bounds show quantitatively that in regimes where the interaction range is small relative to the observation window (or where $p$ varies weakly on the relevant count scale), the dependent thinning behaves close to a Poisson process with the same intensity. The Stein route is particularly informative because it singles out the \emph{short-range} second-order deviation as the essential obstruction to Poisson behaviour, matching the intuition that all dependence is created within distance $2r$.

Neighbour-based thinnings have a long history, with Mat\'ern-type models as the canonical finite-range inhibitory examples \cite{Matern1986,ChiuEtAl2013}. Many generalisations modify the marking/competition mechanism to widen attainable interaction patterns \cite{TeichmannBallaniBoogaart2013,AndersenHahn2016}. Our focus is different: we keep the thinning mechanism as simple as possible (independent marks with an $r$-local retention probability) but replace ``winner selection'' by an explicit dependence on the neighbour count. This makes it possible to (i) compute the pair correlation in closed form for \emph{arbitrary} count rules $p$, and (ii) obtain transparent contact-scale criteria for attraction versus inhibition in terms of $p(0),p(1),p(2)$. On the approximation side, Poisson process approximation via Stein’s method goes back to \cite{BarbourBrown1992} and has been adapted to spatial point processes and dependent thinnings in \cite{Schuhmacher2009,Schuhmacher2009Bernoulli}. Here we exploit the finite-range nature of $T_r$ to obtain a bound that reduces explicitly to the integral of $|g_{Y_r}-1|$ over the interaction range, and we complement this by a direct total-variation coupling bound and a Laplace-functional route.

The paper is organised as follows. Section~\ref{sec:Ta} introduces the thinning $T_r$, its basic properties, and representative examples. Section~\ref{subsec:intensity} derives intensity formulas for Poisson and Cox input, including small-$r$ expansions. Section~\ref{subsec:pair-corr-Ta} provides the exact pair–correlation function for homogeneous Poisson input and the contact-scale asymptotics, with examples illustrating attraction versus repulsion. Section~\ref{sec:poisson-approx-summary} develops the three Poisson-approximation routes (coupling, Laplace functional, Stein), including the small- and moderate-$r$ regimes.

%====================
\section{Neighbour-count dependent thinning}
\label{sec:Ta}

\subsection{Definition and examples}\label{subsec:def-Ta}

We work throughout on \(\mathbb R^d\), \(d\ge 1\). All random objects below are understood to be constructed on a common probability space \((\Omega,\mathcal F,\mathbb P)\), enlarged if necessary. Let \(\mathbf N\) denote the canonical space of all locally finite counting measures on \(\mathbb R^d\), equipped with the vague \(\sigma\)-algebra \(\mathcal N\) \cite[Ch.~2]{LastPenrose2017}. A point process is an \(\mathbf N\)-valued random element
\[
\xi:(\Omega,\mathcal F,\mathbb P)\to(\mathbf N,\mathcal N).
\]
Realisations of simple point processes are identified with locally finite point configurations in \(\mathbb R^d\).

For a Borel set \(B\subset\mathbb R^d\), \(\xi(B)\) denotes the number of points of \(\xi\) in \(B\), and the intensity measure of \(\xi\) is denoted by
\[
\alpha_{\xi}(B):=\mathbb E[\xi(B)].
\]
If \(\alpha_{\xi}\) is absolutely continuous with respect to Lebesgue measure, we denote its density by \(\rho_{\xi}\).

We write \(\mathrm{PPP}(\Lambda)\) for a Poisson point process \(X\) on \(\mathbb R^d\) with intensity measure \(\Lambda\). Thus, for any disjoint bounded Borel sets \(A_1,\dots,A_k\), the random variables \(X(A_1),\dots,X(A_k)\) are independent and
\[
X(A_i)\sim \mathrm{Poisson}\bigl(\Lambda(A_i)\bigr),
\qquad i=1,\dots,k.
\]
When \(\Lambda(dx)=\lambda(x)\,dx\), we write \(\mathrm{PPP}(\lambda)\).

Throughout, we restrict attention to simple point processes, meaning that   $\xi \in  \mathbf N$ is such that
\[
\mathbb P\bigl(\xi(\{x\})\le 1 \text{ for all }x\in\mathbb R^d\bigr)=1.
\]
In the general setting of Last and Penrose \cite[Proposition~6.9]{LastPenrose2017}, a Poisson process with \(s\)-finite intensity measure is simple if and only if its intensity measure is diffuse. Whenever we write \(X\sim\mathrm{PPP}(\Lambda)\), we assume that \(\Lambda\) is locally finite and diffuse. In particular, if
\[
\Lambda(dx)=\lambda(x)\,dx,
\qquad \lambda\in L^1_{\mathrm{loc}}(\mathbb R^d),
\]
then \(\Lambda\) is diffuse and \(\mathrm{PPP}(\Lambda)\) is simple. Here \(L^1_{\mathrm{loc}}(\mathbb R^d)\) denotes the space of locally integrable functions on \(\mathbb R^d\).

Throughout, \(B(x,r)\) denotes the closed Euclidean ball of radius \(r>0\) centred at \(x\in\mathbb R^d\), and
\[
v_d:=|B(0,1)|
\]
denotes the \(d\)-dimensional Lebesgue volume of the unit ball. We write \(|A|\) for the Lebesgue measure of a Borel set \(A\subset\mathbb R^d\).

 For \(\xi\in\mathbf N\) and \(x\in\mathbb R^d,\) \(r>0\), define
\[
n_r(x;\xi):=\xi\bigl(B(x,r)\setminus\{x\}\bigr).
\]
If \(x\) is a point of a simple configuration \(\xi\), then \(n_r(x;\xi)\) is the number of neighbours of \(x\) in the closed ball \(B(x,r)\).

\begin{definition}[Neighbour-count dependent thinning]
\label{def:Ta}
Fix \(r>0\) and a measurable \textit{retention rule}
\(
p:\mathbb N_0\to[0,1].
\) Let \(X=\sum_i\delta_{x_i}\) be a simple point process on \(\mathbb R^d\). Let
\[
\widehat X=\sum_i \delta_{(x_i,U_i)}
\]
be the independently marked point process with ground process \(X,\)
where, conditionally on \(X\), the marks \(U_i\) are independent and uniformly distributed on \([0,1]\).
\textit{The neighbour-count dependent thinning} of \(X\) is the random counting measure \(T_r(X)\) defined by
\[
T_r(X)(B)
:=\int_{B\times[0,1]}
\mathbf 1\!\left\{u\le p\bigl(n_r(x;X)\bigr)\right\}
\,\widehat X(d(x,u)),
\qquad B\in\mathcal B(\mathbb R^d).
\]
Equivalently,
\[
T_r(X)
=\sum_i\mathbf 1\!\left\{U_i\le p\bigl(n_r(x_i;X)\bigr)\right\}\delta_{x_i}.
\]
\end{definition}

The marking construction shows that the dependence in \(T_r(X)\) arises only through the map
\[
X\mapsto p(n_r(\cdot;X)).
\]
Throughout the paper we write
\[
\eta_r(x;\xi):=p(n_r(x;\xi)),
\qquad x\in\mathbb R^d,\ \xi\in\mathbf N.
\]
We also use the Poisson mean of the retention rule,
\[
m_p(t):=\mathbb E[p(N_t)]
=e^{-t}\sum_{n=0}^{\infty}p(n)\frac{t^n}{n!},
\qquad t\ge0,
\]
where \(N_t\sim\mathrm{Poisson}(t)\). Conditionally on \(X\), the retention decisions are independent. This is the standard independent marking construction; see \cite[Ch.~5]{LastPenrose2017}. The rule is \(r\)-local: the retention decision at \(x\) depends only on the restriction of \(X\) to \(B(x,r)\).

Before using Campbell--Mecke formulas, we record the measurability of the local neighbour count. This ensures that the retention probability \(\eta_r(x;\xi)\) is a legitimate measurable integrand in the arguments below.
\begin{lemma}[Joint measurability of the neighbour count]\label{lem:meas-nr}
Fix \(r>0\). On
\[
(\mathbb R^d\times\mathbf N,\mathcal B(\mathbb R^d)\otimes\mathcal N),
\]
the map
\[
(x,\xi)\longmapsto n_r(x;\xi)
=\xi\bigl(B(x,r)\setminus\{x\}\bigr)
\]
is measurable. Consequently, for any measurable \(p:\mathbb N_0\to[0,1]\), the map
\[
(x,\xi)\longmapsto p(n_r(x;\xi))
\]
is measurable.
\end{lemma}

\begin{proof}
For fixed \(r>0\), the set
\(
\{(x,y)\in\mathbb R^d\times\mathbb R^d:\|x-y\|\le r\}
\)
is closed. Hence the map
\(
(x,\xi)\mapsto \xi(B(x,r))
\)
is \(\mathcal B(\mathbb R^d)\otimes\mathcal N\)-measurable. Similarly, since the diagonal
\(
\{(x,y)\in\mathbb R^d\times\mathbb R^d:x=y\}
\)
is closed, the map
\(
(x,\xi)\mapsto \xi(\{x\})
\)
is measurable. Therefore
\[
n_r(x;\xi)
=\xi(B(x,r))-\xi(\{x\})
=\xi\bigl(B(x,r)\setminus\{x\}\bigr)
\]
is measurable. The final assertion follows by composition with \(p\).
\end{proof}

\begin{remark}[Equivariance and stationarity]\label{rem:equivariance}
The thinning kernel associated with \(T_r\) is equivariant under Euclidean isometries: for every isometry \(\phi\), the conditional law of \(T_r(\phi\xi)\) is the image under \(\phi\) of the conditional law of \(T_r(\xi)\). Consequently, if \(X\) is stationary, respectively isotropic, then \(T_r(X)\) is stationary, respectively isotropic.
\end{remark}

\begin{example}[Mat\'ern type~I]
\label{ex:maternI}
With \(p(n)=\mathbf 1\{n=0\}\), the transformation \(T_r\) retains precisely the isolated points of \(X\), yielding the classical Mat\'ern type~I hard-core thinning with interaction radius \(r\) (e.g., \cite{Matern1986}).
\end{example}

\begin{example}[Geometric soft-core family]
\label{ex:geom-softcore}
Fix \(q\in[0,1]\) and \(s\in[0,1]\), and set \(p(n)=q\,s^n\), with the convention \(0^0=1\). Then a point with \(n\) neighbours within distance \(r\) is kept with probability decaying geometrically in \(n\). For Poisson input, the case \(s<1\) favours isolated points and leads to inhibitory behaviour, as quantified by the pair-correlation formulas below. The case \(q=1\), \(s=0\), recovers Mat\'ern type~I, while \(s=1\) gives independent thinning with retention probability \(q\).
\end{example}

\begin{example}[Count-favouring retention]
\label{ex:count-favouring}
If \(p\) is non-decreasing and \(p(1)>p(0)\), then points in locally crowded regions are favoured. In the small-radius regime, the pair-correlation formulas in Section~\ref{subsec:pair-corr-Ta} show contact-scale clustering in the generic case \(p(0)>0\), and a stronger singular contact-scale effect when \(p(0)=0<p(1)\).
\end{example}

%===========================
\subsection{Intensity formulas for Poisson and Cox input}
\label{subsec:intensity}

In this subsection we identify the first-order properties of \(T_r(X)\). 
We first give the intensity formula for Poisson input in terms of its intensity measure. 
The absolutely continuous case is then obtained as a direct consequence. 
Finally, we extend the formula to Cox input by conditioning on the directing measure.

When the Poisson input has intensity measure \(\Lambda(dx)=\lambda(x)\,dx\), we write
\[
\mu_x(r):=\int_{B(x,r)}\lambda(y)\,dy.
\]

The next lemma gives the first-order structure of the thinning for both Poisson and Cox input. 
\begin{lemma}[Intensity formula for Poisson and Cox input]
\label{lem:intensity-Poisson-Cox}
Let \(r>0\) and let \(p:\mathbb N_0\to[0,1]\) be a retention rule. 

\begin{enumerate}
\item[(i)] Let \(X\sim\mathrm{PPP}(\Lambda)\), where \(\Lambda\) is a locally finite diffuse measure on \(\mathbb R^d\), and set \(Y_r:=T_r(X)\). Then, for every bounded Borel set \(B\subset\mathbb R^d\),
\[
\alpha_{Y_r}(B)
=\int_B m_p\bigl(\Lambda(B(x,r))\bigr)\,\Lambda(dx).
\]

In particular, if \(\Lambda(dx)=\lambda(x)\,dx\), where
\(\lambda\in L^1_{\mathrm{loc}}(\mathbb R^d)\) and \(\lambda\ge0\), then \(Y_r\) has intensity function
\[
\rho_{Y_r}(x)
=\lambda(x)e^{-\mu_x(r)}
\sum_{n=0}^{\infty}p(n)\frac{\mu_x(r)^n}{n!}.
\]

\item[(ii)] Let \(C\) be a Cox process on \(\mathbb R^d\) directed by a random measure \(\Lambda\). Assume that \(\Lambda\) is almost surely locally finite and diffuse, and that
\(
\mathbb E[\Lambda(B)]<\infty
\)
for every bounded Borel set \(B\subset\mathbb R^d\). Then, for every bounded Borel set \(B\subset\mathbb R^d\),
\[
\alpha_{T_r(C)}(B)
=\mathbb E\left[
\int_B m_p\bigl(\Lambda(B(x,r))\bigr)\,\Lambda(dx)
\right].
\]
\end{enumerate}
\end{lemma}

\begin{proof}
We first prove (i). Recall that we defined
\(
\eta_r(x;\xi)=p(n_r(x;\xi)),
\  x\in\mathbb R^d,\ \xi\in\mathbf N.
\)
By Lemma~\ref{lem:meas-nr}, the map \((x,\xi)\mapsto \eta_r(x;\xi)\) is measurable. 
Using the marking construction in Definition~\ref{def:Ta}, conditioning on \(X\), and then applying the Campbell--Mecke formula, we obtain
\[
\alpha_{Y_r}(B)
=\mathbb E[T_r(X)(B)]
=\mathbb E\left[\int_B \eta_r(x;X)\,X(dx)\right]
=\int_B\mathbb E[\eta_r(x;X+\delta_x)]\,\Lambda(dx).
\]
Since \(\Lambda\) is diffuse,
\[
n_r(x;X+\delta_x)
=(X+\delta_x)\bigl(B(x,r)\setminus\{x\}\bigr)
=X\bigl(B(x,r)\setminus\{x\}\bigr),
\]
and the last random variable has Poisson distribution with parameter
\(
\Lambda(B(x,r)\setminus\{x\})=\Lambda(B(x,r)).
\)
Therefore
\[
\mathbb E[\eta_r(x;X+\delta_x)]
=m_p\bigl(\Lambda(B(x,r))\bigr),
\]
which proves (i). The density formula follows by substituting
\(\Lambda(dx)=\lambda(x)\,dx\).

For (ii), condition on the directing measure \(\Lambda\). Conditionally on \(\Lambda\), the Cox process \(C\) is a Poisson process with intensity measure \(\Lambda\). Applying (i) conditionally gives
\[
\mathbb E[T_r(C)(B)\mid \Lambda]
=\int_B m_p\bigl(\Lambda(B(x,r))\bigr)\,\Lambda(dx).
\]
Taking expectations yields
\[
\alpha_{T_r(C)}(B)
=\mathbb E\left[\int_B m_p\bigl(\Lambda(B(x,r))\bigr)\,\Lambda(dx)\right].
\]
Since \(0\le m_p\le1\), the assumption \(\mathbb E[\Lambda(B)]<\infty\) guarantees finiteness for bounded \(B\).
\end{proof}

%===========================
\section{Two–point correlation function for the transformation of Poisson input: exact formula and small–radius asymptotics}
\label{subsec:pair-corr-Ta}

In this section, we derive an exact expression for the two-point correlation function of $T_r(X),$ when the input $X$ is a homogeneous Poisson process, followed by small–$r$ asymptotics.

When the input is a homogeneous Poisson process \(X\sim\mathrm{PPP}(\lambda)\), we use the notation
\[
Y_r:=T_r(X),\quad
\mu_r:=\lambda v_d r^d,\quad
N_r\sim\mathrm{Poisson}(\mu_r),\quad
\text{ and }\quad
\lambda_r':=\lambda m_p(\mu_r).
\]
Thus \(\lambda_r'\) is the intensity of \(Y_r\).

\subsection{Second-order product densities and exact pair-correlation formula}
\label{subsec:pair-corr-exact}
We denote by \({\xi}^{(n)}\) the $n-$th factorial measure of a point process $\xi$ and its $n-$th factorial moment measure by \(\alpha_{\xi}^{(n)}\) (Ch.~4 \cite{LastPenrose2017}). When \(\alpha_{\xi}^{(n)}\) is absolutely continuous with respect to Lebesgue measure on \((\mathbb R^d)^n\), we denote its density by \(\rho_{\xi}^{(n)}\). If \(\rho_{\xi}\) is the intensity function of \(\xi\), the pair-correlation function is defined, where the denominator is non-zero, by
\[
g_{\xi}(x,y) = \frac{\rho^{(2)}_{\xi}(x,y)}{\rho_{\xi}(x) \rho_{\xi}(y)}.
\]
When a point process is stationary and isotropic, we use the radial pair–correlation ${g}^\circ_{\xi}:[0,\infty)\to[0,\infty)$,
defined by $g_{\xi}^\circ(s):=g_{\xi}(x,y)$ for any $(x,y)$ with $\|x-y\|=s$.

For a simple point process \(\xi\) with intensity measure
\(\alpha_\xi\), we denote by \(\mathbb P^x\) a version of its one-point
Palm distribution and by \(\mathbb E^x\) the corresponding expectation.
It is characterised by the Campbell--Mecke identity
\begin{equation}\label{eq:CM-1fold}
\mathbb E\left[
\int_{\mathbb R^d}
f(x,\xi)\,\xi(dx)
\right]
=
\int_{\mathbb R^d}
\mathbb E^x\bigl[f(x,\xi)\bigr]\,
\alpha_\xi(dx)
\end{equation}
for every nonnegative measurable function
\(
f:\mathbb R^d\times\mathbf N\to[0,\infty).
\)
Inside \(\mathbb E^x\), the symbol \(\xi\) denotes the canonical
configuration distributed according to the one-point Palm law at \(x\).
A random counting measure \(\xi^x\) whose distribution is
\(\mathbb P^x\) will be called a Palm version of \(\xi\) at \(x\).

Similarly, for a simple point process \(\xi\), we denote by \(\mathbb P^{x,y}\) a version of its two-fold Palm distribution and by \(\mathbb E^{x,y}\) the corresponding expectation. It is characterised by the second-order Campbell--Mecke identity
\begin{equation}\label{eq:CM-2fold}
\mathbb E\left[
\int_{(\mathbb R^d)^2}
f(x,y,\xi)\,\xi^{(2)}(dx,dy)\right]
=\int_{(\mathbb R^d)^2}
\mathbb E^{x,y}\bigl[f(x,y,\xi)\bigr]\,\alpha_\xi^{(2)}(d(x,y)),
\end{equation}
for every nonnegative measurable function
\(
f:(\mathbb R^d)^2\times\mathbf N\to[0,\infty).
\)
Here, inside \(\mathbb E^{x,y}\), the symbol \(\xi\) denotes the canonical configuration distributed according to the two-fold Palm law at \((x,y)\) (for Campbell-Mecke identity and Palm distribution see, for example,  \cite[Sec.~4.1]{CoeurjollyMollerWaagepetersen2016}, \cite{Baddeley2007SpatialPointProcesses})

For Poisson input \(X\sim\mathrm{PPP}(\Lambda)\), the Mecke--Slivnyak theorem (\cite{CoeurjollyMollerWaagepetersen2016, Baddeley2007SpatialPointProcesses}) gives
\[
\mathcal L^{x,y}(X)=\mathcal L(X+\delta_x+\delta_y)
\]
for \(\alpha_X^{(2)}\)-almost every \((x,y)\).

We first express the second-order product density of the thinned process in terms of the two-fold Palm distribution of the input process. The proof combines conditional independence of the retention indicators with the second-order Campbell--Mecke formula.
\begin{lemma}[Second-order product density under \(T_r\)]
\label{lem:general-rho2}
Let \(\xi\) be a simple point process on \(\mathbb R^d\) whose second-order product measure
\(\alpha_\xi^{(2)}\) has density \(\rho_\xi^{(2)}\). Fix \(r>0\) and a retention rule
\(p:\mathbb N_0\to[0,1]\), and set
\(
Y_r:=T_r(\xi).
\)
Then, for Lebesgue-a.e. \(x\ne y\),
\[
\rho_{Y_r}^{(2)}(x,y)
=\rho_\xi^{(2)}(x,y)\,
\mathbb E^{x,y}\left[\eta_r(x;\xi)\eta_r(y;\xi)\right].
\]
\end{lemma}
\begin{proof}
Let \(A,B\subset\mathbb R^d\) be bounded Borel sets. By the marking construction defining \(T_r\), conditionally on \(\xi\), the retention indicators attached to distinct points are independent. Hence
\[
\alpha_{Y_r}^{(2)}(A\times B)
=\mathbb E\left[\int_{A\times B}\eta_r(u;\xi)\eta_r(v;\xi)\,\xi^{(2)}(du,dv)\right].
\]
Apply the two-fold Campbell--Mecke identity \eqref{eq:CM-2fold} with
\(
f(u,v;\zeta)
:=\mathbf 1_A(u)\mathbf 1_B(v)\,\eta_r(u;\zeta)\eta_r(v;\zeta).
\)
Then
\[
\alpha_{Y_r}^{(2)}(A\times B)
=\int_{A\times B}\mathbb E^{u,v}\left[\eta_r(u;\xi)\eta_r(v;\xi)\right]
\,\alpha_\xi^{(2)}(d(u,v)).
\]
Since \(\alpha_\xi^{(2)}\) has density \(\rho_\xi^{(2)}\), we obtain
\[
\alpha_{Y_r}^{(2)}(A\times B)
=\int_A\int_B\rho_\xi^{(2)}(u,v)\mathbb E^{u,v}\left[\eta_r(u;\xi)\eta_r(v;\xi)\right]\,du\,dv.
\]
Thus \(\alpha_{Y_r}^{(2)}\) is absolutely continuous with respect to Lebesgue measure, and the displayed density formula follows for Lebesgue-a.e. \(x\ne y\).
\end{proof}

For Poisson input, the Palm representation in the preceding lemma can be made explicit by the Mecke--Slivnyak formula. Decomposing the two neighbourhood balls into their overlap and two disjoint remainders then reduces the joint retention probability to three independent Poisson counts.
\begin{lemma}[Second-order product density under \(T_r\) for inhomogeneous PPP]
\label{lem:rho2-Ta-ihppp}
Let \(X\sim\mathrm{PPP}(\lambda)\) on \(\mathbb R^d\), where
\(\lambda\in L^1_{\mathrm{loc}}(\mathbb R^d)\) and \(\lambda\ge0\). Fix \(r>0\) and a retention rule
\(p:\mathbb N_0\to[0,1]\), and set
\(
Y_r:=T_r(X).
\)
Then, for Lebesgue-a.e. \(x\ne y\),
\[
\rho_{Y_r}^{(2)}(x,y)
=\lambda(x)\lambda(y)\,
\mathbb E\left[p\bigl(n_r(x;X+\delta_x+\delta_y)\bigr)p\bigl(n_r(y;X+\delta_x+\delta_y)\bigr)
\right].
\]
Writing
\[
a(x,y):=\int_{B(x,r)\setminus B(y,r)}\lambda(u)\,du,\qquad
b(x,y):=\int_{B(y,r)\setminus B(x,r)}\lambda(u)\,du,\qquad
c(x,y):=\int_{B(x,r)\cap B(y,r)}\lambda(u)\,du,
\]
with
\[
I_r(x,y):=\mathbf 1\{\|x-y\|\le r\},
\]
we have
\begin{equation}\label{eq:rho2-ihppp-explicit}
\rho_{Y_r}^{(2)}(x,y)
=\lambda(x)\lambda(y)e^{-(a+b+c)}
\sum_{i,j,k\ge0}p(i+k+I_r)\,p(j+k+I_r)\,\frac{a^i}{i!}\frac{b^j}{j!}\frac{c^k}{k!},
\end{equation}
where, in the last display, \(a=a(x,y)\), \(b=b(x,y)\), \(c=c(x,y)\), and \(I_r=I_r(x,y)\).
\end{lemma}

\begin{proof}
By Lemma~\ref{lem:general-rho2},
\[
\rho_{Y_r}^{(2)}(x,y)
=\rho_X^{(2)}(x,y)\mathbb E^{x,y}\left[\eta_r(x;X)\eta_r(y;X)\right].
\]
For the Poisson process \(X\sim\mathrm{PPP}(\lambda)\), the Mecke--Slivnyak theorem gives
\(
\rho_X^{(2)}(x,y)=\lambda(x)\lambda(y)
\)
for Lebesgue-a.e. \(x\ne y\), and under the two-fold Palm law at \((x,y)\), \(\mathcal L^{x,y}(X)=\mathcal L(X+\delta_x+\delta_y)\). Therefore we have
\[
\rho_{Y_r}^{(2)}(x,y)
=\lambda(x)\lambda(y)\,\mathbb E\left[\eta_r(x;X+\delta_x+\delta_y)\eta_r(y;X+\delta_x+\delta_y)\right],
\]
which is the first displayed formula, since \(\eta_r(z;\xi)=p(n_r(z;\xi))\).

For the explicit expression, decompose the two balls into the disjoint regions
\[
A_{x,y}:=B(x,r)\setminus B(y,r),\qquad
A_{y,x}:=B(y,r)\setminus B(x,r),\qquad
C_{x,y}:=B(x,r)\cap B(y,r).
\]
The Poisson counts
\(
X(A_{x,y}),\  X(A_{y,x}),\  X(C_{x,y})
\)
are independent and have Poisson distributions with respective means
\(
a(x,y),\  b(x,y),\  c(x,y).
\)
Moreover, the added point \(y\) contributes to the neighbour count of \(x\) exactly when
\(\|x-y\|\le r\), and the added point \(x\) contributes to the neighbour count of \(y\) under the same condition. Hence
\[
n_r(x;X+\delta_x+\delta_y)
=X(A_{x,y})+X(C_{x,y})+I_r(x,y),
\]
and
\[
n_r(y;X+\delta_x+\delta_y)
=X(A_{y,x})+X(C_{x,y})+I_r(x,y).
\]
Summing over the three independent Poisson counts gives
\[
\begin{aligned}
&\mathbb E\left[
p\bigl(n_r(x;X+\delta_x+\delta_y)\bigr)
p\bigl(n_r(y;X+\delta_x+\delta_y)\bigr)
\right] \\
&\quad =e^{-(a+b+c)}\sum_{i,j,k\ge0}p(i+k+I_r)\,p(j+k+I_r)\,\frac{a^i}{i!}\frac{b^j}{j!}\frac{c^k}{k!}.
\end{aligned}
\]
Multiplying by \(\lambda(x)\lambda(y)\) yields \eqref{eq:rho2-ihppp-explicit}.
\end{proof}

We now specialise to the case where the input is a homogeneous Poisson process $X\sim\mathrm{PPP}(\lambda)$ and identify the pair–correlation function of the thinned process $T_r(X)$ in closed form. The key point is that, for two locations separated by a vector $h$, the joint retention probabilities depend only on the three Poisson counts in the two $r$-balls and their overlap. This yields an explicit representation of $g_{Y_r}$ in terms of the overlap volume of two balls of radius $r$, and shows in particular that no dependence is created beyond range $2r$.
Define the normalized overlap coefficient
\begin{equation}
\label{omega_d}
    \omega_d(t):=\frac{|B(0,1)\cap B(te_1,1)|}{v_d}\in[0,1] \text{ for } t\in[0,2];  \qquad \omega_d(t)=0\ \text{for }t\ge2.
\end{equation}

\begin{theorem}[Exact pair-correlation for \(T_r(\mathrm{PPP}(\lambda))\)]
\label{thm:g-exact-PPP}
Let \(r>0\), \(\lambda>0\), and let \(X\sim\mathrm{PPP}(\lambda)\) be homogeneous on
\(\mathbb R^d\). 
Assume \(\lambda_r'>0\), equivalently \(m_p(\mu_r)>0\).
For \(h\in\mathbb R^d\setminus\{0\}\), write \(t:=\|h\|/r\). Let
\(U,Z^{(1)},Z^{(2)}\) be independent random variables such that
\[
U\sim \mathrm{Poisson}\bigl(\mu_r\omega_d(t)\bigr),
\qquad
Z^{(1)},Z^{(2)}\stackrel{\mathrm{i.i.d.}}{\sim}
\mathrm{Poisson}\bigl(\mu_r(1-\omega_d(t))\bigr).
\]
Then a version of the radial pair-correlation function of \(Y_r\) is given by
\[
g^\circ_{Y_r}(\|h\|)
=\frac{\mathbb E\left[p\bigl(U+Z^{(1)}+\mathbf 1\{\|h\|\le r\}\bigr)p\bigl(U+Z^{(2)}+\mathbf 1\{\|h\|\le r\}\bigr)\right]}
{m_p(\mu_r)^2}.
\]
In particular,
\[
g^\circ_{Y_r}(\|h\|)=1
\qquad\text{whenever } \|h\|\ge 2r.
\]
\end{theorem}

\begin{proof}
Fix \(x\in\mathbb R^d\) and \(h\ne0\), and put \(y=x+h\). By
Lemma~\ref{lem:rho2-Ta-ihppp} with homogeneous intensity \(\lambda\), the relevant three
regions have volumes
\[
|B(x,r)\setminus B(y,r)|
=|B(y,r)\setminus B(x,r)|
=v_d r^d\bigl(1-\omega_d(t)\bigr),
\]
and
\[
|B(x,r)\cap B(y,r)|
=v_d r^d\omega_d(t),\qquad t=\frac{\|h\|}{r}.
\]
Hence
\(
a=b=\mu_r\bigl(1-\omega_d(t)\bigr),\ c=\mu_r\omega_d(t),
\)
and
\(
I_r:=\mathbf 1\{\|h\|\le r\}.
\)
Therefore Lemma~\ref{lem:rho2-Ta-ihppp} gives
\[
\rho_{Y_r}^{(2)}(x,y)
=\lambda^2\mathbb E\left[p(U+Z^{(1)}+I_r)p(U+Z^{(2)}+I_r)\right],
\]
where
\(
U\sim\mathrm{Poisson}\bigl(\mu_r\omega_d(t)\bigr),
\ Z^{(1)},Z^{(2)}\stackrel{\mathrm{i.i.d.}}{\sim}
\mathrm{Poisson}\bigl(\mu_r(1-\omega_d(t))\bigr),
\)
and these variables are independent. This follows from independence of the Poisson counts on
the disjoint regions
\[
B(x,r)\cap B(y,r),\qquad
B(x,r)\setminus B(y,r),\qquad
B(y,r)\setminus B(x,r).
\]

By the intensity formula,
\[
\rho_{Y_r}(x)
=\lambda\,m_p(\mu_r)=\lambda\,\mathbb E[p(N_r)],\qquad x\in\mathbb R^d.
\]
The denominator is positive by assumption. Therefore, by the definition  of $g^\circ(\|h\|)$ for a stationary isotropic process,
\[
g^\circ_{Y_r}(\|h\|)
=\frac{\mathbb E\left[p(U+Z^{(1)}+I_r)p(U+Z^{(2)}+I_r)\right]}
{\bigl(\mathbb E[p(N_r)]\bigr)^2}.
\]

If \(\|h\|\ge 2r\), then \(\omega_d(t)=0\) and \(I_r=0\). Hence \(U=0\) almost surely, while
\(Z^{(1)}\) and \(Z^{(2)}\) are independent Poisson random variables with mean \(\mu_r\). Thus
\(
\mathbb E[p(Z^{(1)})p(Z^{(2)})]
=\mathbb E[p(N_r)]^2,
\)
and consequently
\(
g^\circ_{Y_r}(\|h\|)=1.
\)
\end{proof}

\subsection{Small–radius asymptotics}
We now derive the first-order expansion of the pair-correlation function of \(Y_r=T_r(X)\) as \(r\downarrow0\), where the input is a homogeneous Poisson process \(X\sim\mathrm{PPP}(\lambda)\).
The results in this subsection are obtained as direct small-\(\mu_r\) consequences of the exact formula in Theorem~\ref{thm:g-exact-PPP}.
Throughout this subsection, \(\lambda>0\) and \(p\) are fixed, and \[ \mu_r=\lambda v_d r^d\downarrow0. \] All \(O(\cdot)\) terms are understood as \(r\downarrow0\), equivalently \(\mu_r\downarrow0\), with \(d\), \(\lambda\), and \(p\) fixed. Whenever the argument contains \(t=\|h\|/r\), uniformity in \(t\) is stated explicitly.

We first treat the non-degenerate case \(p(0),p(1)>0\), in which the intensity of \(Y_r\) remains bounded away from zero as \(r\downarrow0\). The proof expands both the joint-retention numerator and the squared intensity denominator in the exact pair-correlation formula to first order in \(\mu_r\).
\begin{corollary}[First-order contact-scale expansion]
\label{lem:g-first-order}
Let \(X\sim\mathrm{PPP}(\lambda)\) on \(\mathbb R^d\), and set \(Y_r:=T_r(X)\).
Assume \(p(0)>0\) and \(p(1)>0\). For \(t\in(0,2]\), put
\(
I_t:=\mathbf 1\{t\le1\}.
\)
Then, for \(\mu_r\le p(0)/4\),
\[
g_{Y_r}^\circ(tr)
=\frac{p(I_t)^2}{p(0)^2}
\Bigg\{1+\mu_r\Bigg[\omega_d(t)\left(1-\frac{p(I_t+1)}{p(I_t)}\right)^2
+2I_t\left(\frac{p(I_t+1)}{p(I_t)}-\frac{p(1)}{p(0)}\right)\Bigg]\Bigg\}
+R_t(\mu_r),
\]
where the remainder is uniform in \(t\in(0,2]\):
\[
\sup_{t\in(0,2]} |R_t(\mu_r)|
\le C(p(0))\,\mu_r^2.
\]
Consequently:
\begin{itemize}
\item[(i)] If \(t\in(1,2]\), then \(I_t=0\) and
\[
g_{Y_r}^\circ(tr)
=1+\mu_r\,\omega_d(t)\left(1-\frac{p(1)}{p(0)}\right)^2
+O(r^{2d}).
\]

\item[(ii)] If \(t\in(0,1]\), then \(I_t=1\) and
\[
g_{Y_r}^\circ(tr)
=\frac{p(1)^2}{p(0)^2}
\Bigg\{1+\mu_r\Bigg[\omega_d(t)\left(1-\frac{p(2)}{p(1)}\right)^2
+2\left(\frac{p(2)}{p(1)}-\frac{p(1)}{p(0)}\right)\Bigg]\Bigg\}
+O(r^{2d}).
\]
\end{itemize}
\end{corollary}

\begin{proof}
Let
\(
\omega:=\omega_d(t),
\)
where \(\omega_d\) is defined in \eqref{omega_d}. By Theorem~\ref{thm:g-exact-PPP},
\[
g_{Y_r}^\circ(tr)
=\frac{\mathbb E\left[p(U+Z^{(1)}+I_t)p(U+Z^{(2)}+I_t)\right]}
{m_p(\mu_r)^2},
\]
where
\(
U\sim\mathrm{Poisson}(\mu_r\omega),
\ 
Z^{(1)},Z^{(2)}\stackrel{\mathrm{i.i.d.}}{\sim}
\mathrm{Poisson}(\mu_r(1-\omega)),
\)
and these random variables are independent.

First, let us consider the numerator of \(g_{Y_r}^\circ(tr).\) Set
\(
S:=U+Z^{(1)}+Z^{(2)}.
\)
Then \(S\sim\mathrm{Poisson}(\theta)\), where
\[
\theta=\mu_r(2-\omega)\in[\mu_r,2\mu_r].
\]
Put
\(
a_I:=p(I_t),
\ b_I:=p(I_t+1).
\)
Then 
\[
\mathbb P(S=0)=e^{-\theta},
\quad
\mathbb P(U=1,S=1)=\mu_r\omega e^{-\theta},
\quad
\mathbb P(Z^{(i)}=1,S=1)=\mu_r(1-\omega)e^{-\theta},
\quad i=1,2.
\]
Moreover, since \(S\sim\mathrm{Poisson}(\theta)\),
\[
\mathbb P(S\ge2)
=1-e^{-\theta}(1+\theta)
=\int_0^\theta se^{-s}\,ds
\le \frac{\theta^2}{2}
\le 2\mu_r^2.
\]
Hence
\[
\begin{aligned}
&\mathbb{E}\big[p(U + Z^{(1)} + I_t)\,p(U + Z^{(2)} + I_t)\big]
=a_I^2\,\mathbb P(S=0) + a_Ib_I\,\mathbb P(Z^{(1)}=1,S=1) + b_Ia_I\,\mathbb P(Z^{(2)}=1,S=1) \\ &\quad + b_I^2\,\mathbb P(U=1,S=1) + \E[\cdot;\,S\ge2]=e^{-\theta}\Big(a_I^2+\mu_r\big[2a_I b_I(1-\omega)+b_I^2\omega\big]\Big)+R_{\ge2},
\end{aligned}
\]
with \(|R_{\ge2}|\le 2\mu_r^2\). Expanding \(e^{-\theta}=1-\theta+R_e\) with \(|R_e|\le \theta^2/2\) and collecting the \(\mu_r\)-terms gives
\[
\mathbb{E}\big[p(U + Z^{(1)} +I_t)\,p(U + Z^{(2)} + I_t)\big]
= a_I^2+\mu_r\Big(2a_I(b_I-a_I)+\omega(b_I-a_I)^2\Big)+R_1(\mu_r,t),
\]
where \(\sup_{t\in(0,2]}|R_1(\mu_r,t)|\le C\,\mu_r^2\) for some absolute constant \(C\).

Next, by the definition of \(m_p\),
\[
m_p(\mu_r)
=p(0)+\mu_r\{p(1)-p(0)\}+R_2(\mu_r),\qquad |R_2(\mu_r)|\le C_1\,\mu_r^2\quad(0<\mu_r\le1),
\]
with universal \(C_1\) (e.g. \(C_1=e/2\)). Since \(p(0)>0\), for \(\mu_r\le p(0)/4\) the quantity \(m_p(\mu_r)\) is bounded away from zero, and Taylor expansion gives \[ \frac{1}{m_p(\mu_r)^2} = \frac{1}{p(0)^2} \left[ 1 - 2\mu_r\frac{p(1)-p(0)}{p(0)} + R_3(\mu_r) \right], \quad |R_3(\mu_r)\leq C_2 \frac{\mu_r^2}{p(0)^2}, \]
uniformly in t, for a universal $C_2$

Multiplying the numerator expansion and the inverse denominator and regrouping gives the stated first–order term; all remainders are \(O(\mu_r^2)\) uniformly in \(t\). The only \(p\)-dependence in the global constant comes from the inversion via \(p(0)\).
\end{proof}

When \(p(0)=0<p(1)\), the intensity of the thinned process is itself of order \(\mu_r\), and the normalisation in the pair-correlation function becomes singular. We therefore expand the numerator and denominator separately, distinguishing the mutual-inclusion region \(0<t\le1\) from the overlap region \(1<t\le2\).
\begin{corollary}[Contact-scale asymptotics when \(p(0)=0<p(1)\)]
\label{lem:g-first-order-p0zero}
Let \(X\sim\mathrm{PPP}(\lambda)\) on \(\mathbb R^d\), and set \(Y_r:=T_r(X)\).
Assume
\(
p(0)=0,
\qquad p(1)=:c>0.
\)
Then, as \(\mu_r\downarrow0\),
\[
\rho_{Y_r}
=\lambda\left[\mu_r c+\mu_r^2\left(\frac{p(2)}{2}-c\right)
+O(\mu_r^3)
\right].
\]
Moreover, the radial pair-correlation function on the contact scale satisfies the following
asymptotics.

\begin{itemize}
\item[(a)] If \(1<t\le2\), then
\begin{equation}\label{eq:p0zero-ring}
g^\circ_{Y_r}(tr)
=\frac{\omega_d(t)}{\mu_r}
+O(1).
\end{equation}

\item[(b)] If \(0<t\le1\), then
\begin{equation}\label{eq:p0zero-contact}
g^\circ_{Y_r}(tr)
=\frac{1}{\mu_r^2}
\left\{1+\mu_r\left[\frac{p(2)}{p(1)}+\omega_d(t)\left(1-\frac{p(2)}{p(1)}\right)^2
\right]
\right\}
+O(1).
\end{equation}
\end{itemize}
\end{corollary}

\begin{proof}
The intensity formula gives
\(
\rho_{Y_r}
=\lambda m_p(\mu_r).
\)
Since \(p(0)=0\), the expansion of \(m_p\) at zero is
\[
m_p(\mu_r)
=e^{-\mu_r}\sum_{k\ge1}p(k)\frac{\mu_r^k}{k!}
=\mu_r p(1)+\mu_r^2\left(\frac{p(2)}{2}-p(1)\right)
+O(\mu_r^3).
\]
This proves the displayed intensity expansion. In particular,
\[
m_p(\mu_r)^2
=\mu_r^2 c^2\left[1+\mu_r\left(\frac{p(2)}{c}-2\right)+O(\mu_r^2)\right],
\]
and hence
\[
\frac{1}{m_p(\mu_r)^2}
=\frac{1}{\mu_r^2c^2}\left[1-\mu_r\left(\frac{p(2)}{c}-2\right)+O(\mu_r^2)\right].
\]

We now use the exact formula from Theorem~\ref{thm:g-exact-PPP}. First assume \(1<t\le2\).
Then \(I_t=0\). Since \(p(0)=0\), the only first-order contribution to the numerator comes from
\[
U=1,\qquad Z^{(1)}=Z^{(2)}=0.
\]
This event has probability
\(
\mu_r\omega_d(t)+O(\mu_r^2),
\)
and the corresponding value of the product is \(p(1)^2=c^2\). Therefore
\[
\mathbb E\left[p(U+Z^{(1)})p(U+Z^{(2)})\right]
=\mu_r\omega_d(t)c^2+O(\mu_r^2).
\]
Multiplying by \(1/m_p(\mu_r)^2\) gives
\[
g^\circ_{Y_r}(tr)
=\frac{\omega_d(t)}{\mu_r}+O(1),
\]
which proves \eqref{eq:p0zero-ring}.

Now assume \(0<t\le1\). Then \(I_t=1\). The arguments in the numerator are
\(
U+Z^{(1)}+1\) and
\(
U+Z^{(2)}+1.
\)
 We have
\[
\mathbb P\left(U=Z^{(1)}=Z^{(2)}=0\right)=1-\mu_r(2-\omega_d(t))+O(\mu_r^2),\]
\[
\mathbb P\left(Z^{(i)}=1\right)=\mu_r(1-\omega_d(t))+O(\mu_r^2),\ 
\mathbb P(U=1)=\mu_r\omega_d(t)+O(\mu_r^2).
\]
Therefore
\[
\begin{aligned}
&\mathbb E\left[p\left(U+Z^{(1)}+1\right)p\left(U+Z^{(2)}+1\right)\right] \\
&\quad =
c^2\left[1-\mu_r(2-\omega)\right]
+2\mu_r(1-\omega)c\,p(2)
+\mu_r\omega p(2)^2
+O(\mu_r^2),
\end{aligned}
\]
where \(\omega=\omega_d(t)\). Equivalently, using \(s=p(2)/c\),
\[
\mathbb E\left[p(U+Z^{(1)}+1)p(U+Z^{(2)}+1)\right]
=c^2\left\{1+\mu_r\left[-(2-\omega)+2(1-\omega)s+\omega s^2\right]\right\}
+O(\mu_r^2).
\]
Dividing by \(m_p(\mu_r)^2\), we obtain
\[
\begin{aligned}
g^\circ_{Y_r}(tr)
&=\frac{1}{\mu_r^2}
\left\{1+\mu_r\left[-(2-\omega)+2(1-\omega)s+\omega s^2-(s-2)\right]\right\}
+O(1) \\
&=\frac{1}{\mu_r^2}\left\{1+\mu_r\left[s+\omega(1-s)^2\right]\right\}
+O(1),
\end{aligned}
\]
which proves \eqref{eq:p0zero-contact}.
\end{proof}

%===========================
\subsection{Examples: attraction vs.\ repulsion at small scales}
\label{subsec:examples-attract-repel}

We illustrate how specific choices of the neighbour–count retention rule $p(n)$ generate \emph{repulsion} ($g^\circ_{Y_r}<1$) or \emph{attraction} ($g^\circ_{Y_r}>1$) at contact scales when the input is $X\sim\mathrm{PPP}(\lambda)$. Throughout, we use the exact representation and the small–$r$ expansions established in Theorem~\ref{thm:g-exact-PPP}, Corollary~\ref{lem:g-first-order} and Corollary~\ref{lem:g-first-order-p0zero}. 

For the reader's convenience, we recall the notations introduced above. We put $Y_r= T_r(X)$
and let $\mu_r:=\lambda v_d r^d\downarrow 0$ and write $t:=\|h\|/r\in(0,2]$ to probe separations on the \emph{contact scale}. Set $\omega:=\omega_d(t)\in[0,1]$ for the normalized overlap volume of two $r$–balls at distance $\|h\|=tr$, and $I:=\mathbf{1}\{t\le 1\}$.

\begin{example}[Mat\'ern type~I: hard--core thinning]
\label{ex:maternI-g}
Take $p(n)=\mathbf{1}\{n=0\}$ and fix $t\in(0,2]$. Then:
\begin{itemize}
\item If $t\le 1$, the mutual inclusion indicator $I=1$ implies both arguments of $p$ in Theorem~\ref{thm:g-exact-PPP} are $\ge 1$, hence
\[
g_{Y_r}^{\circ}(tr)=0
\]
exactly (hard--core with radius $r$).
\item If $1<t<2$, then $I=0$, and Theorem~\ref{thm:g-exact-PPP} gives
\[
g_{Y_r}^\circ(tr)=\exp\{\mu_r\omega_d(t)\}
=1+\mu_r\omega_d(t)+O(\mu_r^2)>1.
\]
At $t=2$, since $\omega_d(2)=0$, one has
\[
g_{Y_r}^{\circ}(2r)=1.
\]
Thus Mat\'ern~I exhibits strict \emph{repulsion} on $(0,r]$ and a classical \emph{ring attraction} on $(r,2r)$; see also \cite[§5.4]{ChiuEtAl2013}.
\end{itemize}
\end{example}

\begin{example}[Soft–core, geometric retention]\label{ex:geom}
Let $p(n)=q\,s^{\,n}$ with $q\in(0,1]$ and $s\in(0,1)$. Then
$p(0)=q$, $p(1)=q s$, $p(2)=q s^2$, and hence
\[
\frac{p(1)}{p(0)}=s,\qquad \frac{p(2)}{p(1)}=s.
\]

Using Corollary~\ref{lem:g-first-order}:

{(Ring, $1<t<2$).} Here $I=0$, so Corollary~\ref{lem:g-first-order}(i) yields
\[
g_{Y_r}^\circ(tr)
\;=\;1+\mu_r\,\omega_d(t)\,(1-s)^2\;+\;O(\mu_r^2)\;>\;1,
\]
In words, after short-range inhibition there is an enhanced likelihood of finding another point at distances just beyond $r$, a well-known feature of Mat\'ern-type and soft-core thinnings \cite[§5.4]{ChiuEtAl2013}, \cite[§6.3]{MollerWaagepetersen2003}.

{(Contact, $0<t\le 1$).} Here $I=1$, and Corollary~\ref{lem:g-first-order}(ii) gives
\[
g_{Y_r}^\circ(tr)
\;=\;\Big(\tfrac{p(1)}{p(0)}\Big)^2
\Big\{1+\mu_r\big[\omega_d(t)\big(1-\tfrac{p(2)}{p(1)}\big)^2
+2\big(\tfrac{p(2)}{p(1)}-\tfrac{p(1)}{p(0)}\big)\big]\Big\}+O(\mu_r^2).
\]
Since $p(2)/p(1)=p(1)/p(0)=s$, the linear term outside $\omega_d(t)$ cancels, and we obtain
\[
g_{Y_r}^\circ(tr)\;=\;s^2\Big\{1+\mu_r\,\omega_d(t)\,(1-s)^2\Big\}+O(\mu_r^2)\;<\;1
\]
for all sufficiently small $\mu_r$ (because $s^2<1$). Thus there is \emph{repulsion at contact} ($\|h\|\le r$).

In summary, the geometric soft–core rule produces the canonical “repulsion at contact, attraction on the ring” pattern. The special case $s=1$ reduces to independent thinning ($g_{Y_r}\equiv 1$).
\end{example}

\begin{example}[Count--favouring retention near contact]
\label{ex:increasing-local}
Assume $p(0)>0$ and $p(1)>p(0)$. Then, keeping $\mu_r=\lambda v_d r^d\downarrow 0$:
\begin{itemize}
\item For $0<t\le 1$ ($I=1$), Corollary~\ref{lem:g-first-order}(ii) gives
\[
g_{Y_r}^\circ(tr)
=\Big(\tfrac{p(1)}{p(0)}\Big)^{2}\Big\{1+\mu_r\Big[\omega_d(t)\Big(1-\tfrac{p(2)}{p(1)}\Big)^2
+2\Big(\tfrac{p(2)}{p(1)}-\tfrac{p(1)}{p(0)}\Big)\Big]\Big\}+O(\mu_r^2).
\]
Hence $g_{Y_r}^\circ(tr)>1$ for all sufficiently small $\mu_r$ (since $(p(1)/p(0))^2>1$ and the bracket equals $1+O(\mu_r)$). This yields \emph{attraction} at contact generated purely by the count--favouring rule.
\item For $1<t\le 2$ ($I=0$), Corollary~\ref{lem:g-first-order}(i) gives
\[
g_{Y_r}^\circ(tr)=1+\mu_r\,\omega_d(t)\Big(1-\tfrac{p(1)}{p(0)}\Big)^{2}+O(\mu_r^2)\;\ge\;1,
\]
with equality at first order iff $p(1)=p(0)$ (see Example~\ref{ex:flat}).
\end{itemize}
\end{example}

\begin{example}[Suppressing the ring: flat rule at $n=0,1$]
\label{ex:flat}
Suppose $p(1)=p(0)>0$ (flat to first order at the origin of the count scale). Then, for $1<t\le 2$ ($I=0$),
\[
g_{Y_r}^\circ(tr)=1+O(\mu_r^2)
\]
by Corollary~\ref{lem:g-first-order}(i), so the \emph{first--order} ring overshoot vanishes.
If, in addition, $p(2)=p(1)=p(0)$, then Corollary~\ref{lem:g-first-order}(ii) gives
$g_{Y_r}^\circ(tr)=1+O(\mu_r^2)$ also for $0<t\le 1$.
If $p(n)\equiv p$ for all $n$ (independent thinning), then $T_r(X)$ is Poisson and $g_{Y_r}\equiv 1$ exactly.
\end{example}

\begin{example}[Cluster–favouring]
\label{ex:cluster-favoring}
Let $p(n)=1-e^{-\alpha n}$ with $\alpha>0$. Then $p(0)=0$, $c:=p(1)=1-e^{-\alpha}\in(0,1)$, and $p(2)=1-e^{-2\alpha}=c(2-c)$, so
\[
s:=\frac{p(2)}{p(1)}=1+e^{-\alpha}.
\]
\emph{Intensity.} Using $\E[1-e^{-\alpha N}]=1-\exp(\mu_r(e^{-\alpha}-1))$ for $N\sim\mathrm{Poisson}(\mu_r)$,
\[
\rho_{Y_r}=\lambda\Big(1-\exp \big(\mu_r(e^{-\alpha}-1)\big)\Big)
=\lambda\Big(\mu_r c+\mu_r^2\Big(\tfrac12 p(2)-c\Big)+O(\mu_r^3)\Big),\qquad \mu_r=\lambda v_d r^d.
\]

\emph{Radial pair correlation (contact scale $\|h\|=tr$).}
By Corollary~\ref{lem:g-first-order-p0zero}:
\begin{itemize}
\item If $1<t\le 2$ ($I=0$),
\[
g^\circ_{Y_r}(tr)=\frac{\omega_d(t)}{\mu_r}+O(1),\qquad \mu_r\downarrow 0.
\]
\item If $0<t\le 1$ ($I=1$),

\begin{align*}
 g^\circ_{Y_r}(tr)
&=\frac{1}{\mu_r^{2}}\Big\{\,1+\mu_r\big[s+\omega_d(t)\,(1-s)^2\big]\Big\}+O(1)\\
&=\frac{1}{\mu_r^{2}}\Big\{\,1+\mu_r\big[\,1+e^{-\alpha}+\omega_d(t)\,e^{-2\alpha}\big]\Big\}+O(1).   
\end{align*}

\end{itemize}
Thus $p(n)=1-e^{-\alpha n}$ produces pronounced \emph{clustering} for $\|h\|\le 2r$, with a contact peak of order $\mu_r^{-2}$ on $[0,r]$ and a ring level of order $\mu_r^{-1}$ on $(r,2r]$.
\end{example}

\begin{remark}[Comparison with Mat\'ern thinning]
\label{rem:compare-matern}
Example~\ref{ex:maternI-g} reproduces the classical Mat\'ern~I behaviour: \emph{complete repulsion} for $\|h\|<r$ and a \emph{ring} with $g>1$ on $(r,2r)$ (see \cite[§5.4]{ChiuEtAl2013}). Example~\ref{ex:geom} generalises this to soft–core rules $p(n)=q\,s^n$: we still obtain $g<1$ at contact and $g>1$ in the ring, with amplitudes modulated by $s$. In contrast, Example~\ref{ex:increasing-local} shows a phenomenon \emph{not} present in Mat\'ern I/II: by choosing a locally increasing count–favouring rule with $p(1)>p(0)$ (keeping $p$ bounded), one obtains $g>1$ already at contact scales ($\|h\|\le r$) for sufficiently small $r$, while retaining the ring overshoot on $(r,2r)$. This attraction is driven by the thinning rule itself, rather than by an external clustering mechanism.
\end{remark}

%===========================
\section{Poisson approximation by complementary routes}\label{sec:poisson-approx-summary}
\subsection{Preliminary definitions and basic properties}
Throughout this section, \(W\subset\mathbb R^d\) is a compact window. We equip \(W\) with the metric \[ d_0(x,y):=1\wedge\|x-y\|, \qquad x,y\in W. \] Then \((W,d_0)\) is a compact metric space with \(d_0\le1\), and \(d_0\) induces the Euclidean relative topology on \(W\).

We  work on the subspace $\mathbf{N}(W):=\{\xi\in\mathbf{N}:\xi(W^{\mathrm c})=0\}$ of finite counting
measures supported in a window $W$, endowed with the canonical $\sigma$–field
\[
\mathcal{N}_W:=\sigma\{\ \xi\mapsto \xi(B):\ B\in\mathcal{B}(W)\ \}.
\]
For a counting measure \(\xi\in\mathbf N\) and a Borel set \(A\subset\mathbb R^d\), we write \(\xi_A \) for the restriction of \(\xi\) to \(A\), defined by \[ \xi_A(B):=\xi(B\cap A), \qquad B\in\mathcal B(\mathbb R^d). \]

As before, we consider $X\sim\mathrm{PPP}(\lambda)$ on $\mathbb{R}^d$, $r>0$ is fixed, and $ T_r$ is the neighbour–count dependent thinning from Section~\ref{sec:Ta} with measurable $p:\mathbb{N}_0\to[0,1]$.
Recall the notations, 
for a homogeneous Poisson process \(X\sim\mathrm{PPP}(\lambda)\), we use 
\[
Y_r:=T_r(X),\qquad
\mu_r:=\lambda v_d r^d,\qquad
N_r\sim\mathrm{Poisson}(\mu_r),
\qquad
\lambda_r':=\lambda m_p(\mu_r).
\]
Thus \(\lambda_r'\) is the intensity of \(Y_r\).

Let \(\Pi_r\sim\mathrm{PPP}(\lambda_r')\) be a homogeneous Poisson
process with the same intensity as \(Y_r\), and denote the
restrictions of \(Y_r\) and \(\Pi_r\) to \(W\) by
\[
Y_{r,W}:=(Y_r)_W,
\qquad
\Pi_{r,W}:=(\Pi_r)_W.
\]
We develop three complementary approaches to approximate the windowed thinning \(Y_{r,W}\) by a Poisson process \(\Pi_{r,W}\) with matched intensity \(\lambda_r'\). 
Each route works under the same \(r\)-local construction but quantifies the error in a \emph{different metric}, leading to slightly different upper bounds and regimes of sharpness:

(i) a \emph{direct coupling} that yields total variation control \(d_{\mathrm{TV}}\) on \(\mathbf N(W)\) (via Campbell–Mecke/Palm calculus and standard coupling inequalities; \cite{DaleyVereJones2008,LastPenrose2017,Lindvall1992});

(ii) a \emph{Laplace–functional} discrepancy, exploiting that all dependence is confined to pair distances \(\le 2r\);

 (iii) a \emph{Stein bound} in the Barbour–Brown \(d_2\) Wasserstein–type distance induced by the matching metric \(d_1\) on configurations, which reveals an explicit dependence on the short–range pair correlation \(g_{Y_r}\) (\cite{BarbourBrown1992,Schuhmacher2009,Schuhmacher2009Bernoulli}).

\begin{definition}[Total variation on $\mathbf{N}(W)$]\label{def:TV-W}
For probability laws $P,Q$ on $(\mathbf{N}(W),\mathcal{N}_W)$ we define
\[
d_{\mathrm{TV}}(P,Q):=\sup_{A\in\mathcal{N}_W}|P(A)-Q(A)|.
\]
\end{definition}

The equivalent coupling characterisation of the total variation distance is (\cite[Thm. (5.2)]{Lindvall1992})
\[
d_{\mathrm{TV}}(P,Q)
=\inf\Big\{\mathbb{P}\{\xi\neq\zeta\}:\ (\xi,\zeta)\ \text{is a coupling of }(P,Q)\Big\}.
\]
In particular, for any given coupling $(\xi,\zeta)$ with marginals $P$ and $Q$,
\begin{equation}\label{eq:coupling-ineq-TV}
d_{\mathrm{TV}}(P,Q)\ \le\ \mathbb{P}\{\xi\neq \zeta\}.
\end{equation}

Using Stein's method we compare the laws $\mathcal{L}(Y_{r,W})$ and $\mathcal{L}(\Pi_{r,W})$ in the Wasserstein-type $d_2$–metric on finite configurations of $W$ induced by the $d_1$ matching distance (Barbour–Brown \cite{BarbourBrown1992}; cf.\ \cite[§2]{Schuhmacher2009Bernoulli}). Let \(S_n\) be the symmetric group of permutations of
\(\{1,\dots,n\}\).

\begin{definition}[Barbour--Brown $d_1$ and induced $d_2$ on $\mathbf{N}(W)$]
For $\xi=\sum_{i=1}^n\delta_{x_i},\ \zeta=\sum_{j=1}^m\delta_{y_j}\in\mathbf{N}(W)$ set
\[
d_1(\xi,\zeta)=
\begin{cases}
\displaystyle \min_{\pi\in S_n}\frac{1}{n}\sum_{i=1}^n d_0(x_i,y_{\pi(i)}), & n=m\ge 1,\\[1ex]
1, & n\ne m,\\
0, & n=m=0.
\end{cases}
\]
For laws $P,Q$ on $\mathbf{N}(W)$ define
\begin{align*}
    d_2(P,Q)
    :=\sup\Big\{\big|\E_P f-\E_Q f\big|, &\text{ for }\ f:\mathbf{N}(W)\to[0,1],
|f(\xi)-f(\zeta)|\le d_1(\xi,\zeta)\ \forall\,\xi,\zeta\in\mathbf{N}(W)\Big\}.
\end{align*}

\end{definition}

To apply the local–dependence property of $Y_{r,W}$, we need to identify the maximal distance up to which dependence may occur (equivalently, the minimal distance beyond which independence holds).
For a point process $\xi$ on $\mathbb{R}^d$ and Borel sets $A$, let $\mathcal{F}_{\xi}(A)$ be the $\sigma$–field generated by $\xi_A$.
We write $A^{r}:=\{x:\ \mathrm{dist}(x,A)\le r\}$ for the $r-$ neighbourhood of $A.$

\begin{lemma}[Finite–range independence for Poisson input]\label{lem:finite-range}
Let $X\sim\mathrm{PPP}(\lambda)$ on $\R^d$, and let $Y_r=T_r(X).$ For bounded Borel sets $A,B\subset\R^d$ with $\mathrm{dist}(A,B)>2r$,
\[
\mathcal{F}_{Y_r}(A)\ \text{ and }\ \mathcal{F}_{Y_r}(B)\quad\text{are independent.}
\]
\end{lemma}

\begin{proof}
Recall that
\[
\widehat X=\sum_i \delta_{(x_i,U_i)}
\]
is the independently marked version of \(X\) used in the construction of \(T_r(X)\).
If \(\operatorname{dist}(A,B)>2r\), then
\(A^r\cap B^r=\varnothing\). Moreover,
\[
\mathcal F_{Y_r}(A)
\subseteq\sigma\!\left(\widehat X_{A^r\times[0,1]}\right),
\qquad \mathcal F_{Y_r}(B)\subseteq\sigma\!\left(\widehat X_{B^r\times[0,1]}\right).
\]
Since \(\widehat X\) is a Poisson process and
\((A^r\times[0,1])\cap(B^r\times[0,1])=\varnothing\), the two
sigma-fields on the right are independent. Hence
\(\mathcal F_{Y_r}(A)\) and \(\mathcal F_{Y_r}(B)\) are independent.
\end{proof}

Since \(Y_r\) is a thinning of a homogeneous Poisson process, all its
factorial moment measures are absolutely continuous. We denote the
density of its \(n\)-th factorial cumulant measure by
\(\kappa_{Y_r}^{(n)}\).

We next convert finite-range dependence into an integrated bound for all factorial cumulant densities. This estimate will justify the factorial-cumulant expansion of the log--Laplace functional and will also control the higher-order remainder in the sharpness argument. The proof uses factorisation over the connected components of the \(2r\)-neighbourhood graph, bounds the resulting cumulant by a sum over spanning trees, and integrates each tree successively from its leaves to its root.

\begin{lemma}[Integrated factorial-cumulant bound]
\label{lem:integrated-cumulant-bound}
Let \(Y_r=T_r(X)\), where \(X\sim\mathrm{PPP}(\lambda)\), and put
\(V_r:=v_d(2r)^d.\)
Then, for every \(n\ge2\) and every compact Borel set
\(W\subset\R^d\),
\begin{equation}
\label{eq:integrated-cumulant-bound}
\int_{W^n}
\left|\kappa_{Y_r}^{(n)}(x_1,\dots,x_n)\right|
\,  dx_1\cdots  dx_n
\le |W|\lambda^n2^{n-1}V_r^{n-1}n^{n-2}.
\end{equation}
\end{lemma}

\begin{proof}
Fix \(n\ge2\), and write
\[
[n]:=\{1,\dots,n\}.
\]
For every nonempty \(I\subset[n]\), let
\[
x_I:=(x_i)_{i\in I},
\qquad  q_I(x_I)
 :=\lambda^{-|I|}
\rho_{Y_r}^{(|I|)}(x_I).
\]
By the multivariate Mecke formula and the construction of the thinning,
\[
0\le q_I(x_I)\le1
\]
for Lebesgue-a.e. \(x_I\).

Let \(G\) be the graph with vertex set \([n]\) in which distinct
vertices \(i\) and \(j\) are adjacent whenever
\[
\|x_i-x_j\|\le2r.
\]
For \(I\subset[n]\), let \(G[I]\) denote the subgraph of \(G\) induced
by \(I\). If \(C_1,\dots,C_m\) are the vertex sets of the connected
components of \(G[I]\), then
\begin{equation}
\label{eq:q-component-factorization}
 q_I(x_I)
 =\prod_{\ell=1}^m q_{C_\ell}(x_{C_\ell}).
\end{equation}
Indeed, the sets
\[
\bigcup_{i\in C_\ell}B(x_i,r),
\qquad 1\le\ell\le m,
\]
are pairwise disjoint, and the corresponding restrictions of the
driving Poisson process are independent.

Let \(\mathcal P([n])\) denote the set of partitions of \([n]\). For
\(\pi\in\mathcal P([n])\), let \(|\pi|\) be the number of blocks of
\(\pi\), and put
\[
\mu(\pi)
:=(-1)^{|\pi|-1}(|\pi|-1)!,
\qquad M_\pi := \prod_{B\in\pi}q_B(x_B).
\]
By the factorial moment--cumulant relation
\cite[Eq.~(5.5.9)]{DaleyVereJones2003},
\begin{equation}
\label{eq:normalized-factorial-cumulant}
\lambda^{-n}
\kappa_{Y_r}^{(n)}(x_1,\dots,x_n)
=\sum_{\pi\in\mathcal P([n])}\mu(\pi)M_\pi.
\end{equation}

For \(\pi\in\mathcal P([n])\), let \(\varphi_G(\pi)\) be the refinement
of \(\pi\) obtained by replacing each block \(B\in\pi\) by the vertex
sets of the connected components of \(G[B]\). By
\eqref{eq:q-component-factorization},
\[
M_\pi=M_{\varphi_G(\pi)}.
\]

For a finite graph \(H\), write \(\operatorname{ST}(H)\) for the number
of spanning trees of \(H\), with the convention
\(\operatorname{ST}(H)=0\) when \(H\) is disconnected. The regrouping
argument leading to
\cite[Eq.~(9.8)]{FerayMeliotNikeghbali2016}, followed by
\cite[Cor.~9.3.3 and Eq.~(9.5)]{FerayMeliotNikeghbali2016}, applied to
\eqref{eq:normalized-factorial-cumulant}, gives
\[
\left|\lambda^{-n}\kappa_{Y_r}^{(n)}(x_1,\dots,x_n)\right|
\le 2^{n-1}\operatorname{ST}(G),
\]
where we used \(0\le M_\pi\le1\).

Let \(\mathcal T_n\) denote the collection of labelled trees with vertex
set \([n]\), and write \(E(T)\) for the edge set of
\(T\in\mathcal T_n\). Since a tree \(T\in\mathcal T_n\) is a spanning
tree of \(G\) precisely when all its edges belong to \(G\),
\[
\operatorname{ST}(G)
=\sum_{T\in\mathcal T_n}\prod_{\{i,j\}\in E(T)}\mathbf 1\{\|x_i-x_j\|\le2r\}.
\]
Consequently,
\begin{equation}
\label{eq:factorial-cumulant-tree-bound}
\left|\kappa_{Y_r}^{(n)}(x_1,\dots,x_n)\right|
\le 2^{n-1}\lambda^n
\sum_{T\in\mathcal T_n}\prod_{\{i,j\}\in E(T)}
\mathbf 1\{\|x_i-x_j\|\le2r\}.
\end{equation}

Fix \(T\in\mathcal T_n\), choose one of its vertices as the root, and
integrate successively over its leaves. The root contributes at most
\(|W|\), while every other vertex is constrained to lie in a ball of
radius \(2r\) centred at its parent and therefore contributes at most
\(V_r\). Hence
\[
\int_{W^n}\prod_{\{i,j\}\in E(T)}
\mathbf 1\{\|x_i-x_j\|\le2r\}
\, d x_1\cdots d x_n
\le |W|V_r^{n-1}.
\]
Integrating \eqref{eq:factorial-cumulant-tree-bound} and using Cayley's
formula
\[
|\mathcal T_n|=n^{n-2}
\]
gives
\[
\int_{W^n}\left|\kappa_{Y_r}^{(n)}(x_1,\dots,x_n)
\right|\, d x_1\cdots d x_n
\le |W|\lambda^n2^{n-1}V_r^{n-1}n^{n-2},
\]
as claimed.
\end{proof}

\subsection{Poisson approximation using coupling and Laplace functionals}
\label{subsec:nonstein}

We first compare \(Y_{r,W}\) with the restriction of a Poisson process having
the same intensity. The comparison is obtained by coupling the
neighbour-count dependent thinning with an independent thinning of the same
Poisson input, using the same auxiliary marks.
Under this coupling, the two processes can differ only at input points for which the two retention thresholds disagree, so the total variation distance is bounded by the expected number of such points.
\begin{theorem}[Direct coupling to independent thinning]
\label{thm:coupling-TV}
Let \(X\sim\mathrm{PPP}(\lambda)\) be homogeneous on \(\mathbb R^d\), and let
\(W\subset\mathbb R^d\) be a compact Borel set. Then
\[
d_{\mathrm{TV}}
\left( \mathcal L(Y_{r,W}), \mathcal L(\Pi_{r,W})\right)
\le \lambda |W|\, \mathbb E \left|
p(N_r)-m_p(\mu_r) \right|.
\]
\end{theorem}

\begin{proof}
Let
\[
\widehat X=\sum_i\delta_{(x_i,U_i)}
\]
be the independently marked Poisson process used in the definition of
\(Y_r=T_r(X)\). Thus
\[
Y_r(B)= \int_{B\times[0,1]}
\mathbf 1\{u\le \eta_r(x;X)\}
\,\widehat X(d(x,u)),
\qquad B\in\mathcal B(\mathbb R^d).
\]

Using the same marked process, define
\[
\widetilde\Pi_r(B)
:=\int_{B\times[0,1]} \mathbf 1\{u\le m_p(\mu_r)\}
\,\widehat X(d(x,u)).
\]
By the independent thinning theorem,
\(
\widetilde\Pi_r\sim\mathrm{PPP}(\lambda_r'),
\, \lambda_r'=\lambda m_p(\mu_r).
\)
Consequently,
\(\mathcal L((\widetilde\Pi_r)_W)
=\mathcal L(\Pi_{r,W}).
\)

Let
\[
D_W:=\int_{W\times[0,1]}\left|\mathbf 1\{u\le\eta_r(x;X)\}-
\mathbf 1\{u\le m_p(\mu_r)\}
\right|\,\widehat X(d(x,u)).
\]
Thus \(D_W\) counts the points of \(X\) in \(W\) for which the two thinning
decisions differ. In particular,
\[
\{Y_{r,W}\ne(\widetilde\Pi_r)_W\}=\{D_W\ge1\}.
\]
The coupling inequality and Markov's inequality therefore give
\[
d_{\mathrm{TV}}
\left(\mathcal L(Y_{r,W}),\mathcal L(\Pi_{r,W})\right)
\le \mathbb P\{Y_{r,W}\ne(\widetilde\Pi_r)_W\}
=\mathbb P\{D_W\ge1\} \le \mathbb E[D_W].
\]

Conditionally on \(X\), the marks are independent and uniformly distributed
on \([0,1]\). Hence
\[
\mathbb E[D_W\mid X]
=\int_W\left|\eta_r(x;X)-m_p(\mu_r)\right|\,X(dx).
\]

Taking expectations and applying the Mecke formula, we obtain
\[
\mathbb E[D_W]
= \mathbb E\left[\int_W\left|\eta_r(x;X)-m_p(\mu_r)\right|\,X(dx)\right] 
=\lambda\int_W\mathbb E\left[\left|\eta_r(x;X+\delta_x)-m_p(\mu_r)\right|\right]\,dx.
\]
Since
\[
n_r(x;X+\delta_x)
=(X+\delta_x)\bigl(B(x,r)\setminus\{x\}\bigr)
=X\bigl(B(x,r)\setminus\{x\}\bigr),
\]
and \(X\) is a Poisson process with diffuse intensity measure
\(\lambda\,dx\), the last random variable has Poisson distribution with
parameter
\(
\lambda\bigl|B(x,r)\setminus\{x\}\bigr|
=\lambda |B(x,r)|=\mu_r.
\)
Therefore,
\[
\mathbb E[D_W]
=\lambda |W|\,\mathbb E\left|p(N_r)-m_p(\mu_r)\right|,
\qquad N_r\sim\mathrm{Poisson}(\mu_r).
\]
The claimed estimate follows.
\end{proof}

The coupling estimate becomes explicit when \(p\) is discretely Lipschitz. Indeed, comparison with \(p(0)\) bounds the fluctuation of \(p(N_r)\) by \(L_pN_r\), whose expectation is \(L_p\mu_r\).
\begin{corollary}[Small-\(\mu_r\) bound]
\label{cor:small-mu-TV}
Let
\(
L_p:=\sup_{n\ge0}|p(n+1)-p(n)|.
\)
Then
\[
d_{\mathrm{TV}}
\left(\mathcal L(Y_{r,W}),\mathcal L(\Pi_{r,W})\right)
\le 2L_p\lambda |W|\mu_r = 2L_pv_d|W|\lambda^2r^d.
\]
\end{corollary}

\begin{proof}
For every \(n\ge0\),
\[
|p(n)-p(0)|
\le \sum_{k=0}^{n-1}|p(k+1)-p(k)|
\le L_p n.
\]
Consequently,
\(
\mathbb E|p(N_r)-p(0)|
\le L_p\mathbb E[N_r] =L_p\mu_r,
\)
and
\(
|m_p(\mu_r)-p(0)|
= \left|\mathbb E[p(N_r)-p(0)]\right|
\le L_p\mu_r.
\)
Thus
\[
\mathbb E|p(N_r)-m_p(\mu_r)|
\le 2L_p \mu_r.
\]
The result follows from Theorem~\ref{thm:coupling-TV}.
\end{proof}

For a point process \(\xi\) and a bounded measurable function
\(g:\R^d\to[0,\infty)\) with bounded support, we write
\[
L_\xi(g)
:= \E\left[\exp\left(-\int_{\R^d}g(x)\,\xi(dx)\right)\right]
\]
for its Laplace functional.

The next result compares \(Y_r\) and the matched Poisson process through their log--Laplace functionals. The proof expands the logarithm in factorial cumulants, uses Lemma~\ref{lem:integrated-cumulant-bound} to justify the expansion at the required argument, and observes that the first-order terms cancel because the two processes have the same intensity.

\begin{theorem}[Log--Laplace-functional control]
\label{thm:laplace-control}
Let \(W\subset\R^d\) be a compact Borel set, and let
\(g:\R^d\to[0,\infty)\) be bounded and measurable, with
\(\operatorname{supp}(g)\subset W\). There exist constants
\(c_d,C_d\in(0,\infty)\), depending only on \(d\), such that, whenever
\(
\lambda r^d\le c_d,
\)
one has
\begin{equation}
\label{eq:Laplace-main}
\left| \log L_{\Pi_r}(g)-\log L_{Y_r}(g)\right|
\le C_d\lambda^2\|g\|_\infty |W|(2r)^d.
\end{equation}
\end{theorem}

\begin{proof}
Put
\[
 h(x):=1-e^{-g(x)}.
\]
Then \(0\le h\le1\), \(\operatorname{supp}(h)\subset W\), and
\(
\|h\|_\infty\le \min\{1,\|g\|_\infty\}.
\)

We first justify the factorial-cumulant expansion. For \(z\) in a
neighbourhood of the origin, set
\[
\Phi(z)
:=\E\left[\exp\left(\int_W\log(1-zh(x))\,Y_r(dx)\right)\right].
\]
Since \(Y_r\) is a thinning of \(X\), its factorial product densities
satisfy
\(
 0\le \rho_{Y_r}^{(n)}\le\lambda^n.
\)
Consequently, the factorial-moment expansion
\[
\Phi(z)
= 1+\sum_{n=1}^\infty \frac{(-z)^n}{n!}\int_{W^n}\rho_{Y_r}^{(n)}(x_1,\dots,x_n)
\prod_{i=1}^n h(x_i)\,   dx_1\cdots  dx_n
\]
converges absolutely for every \(z\in\mathbb C\), and hence defines an
entire function.

By the generating functional expansion and the factorial
moment--cumulant relation
\cite[Section~5.5]{DaleyVereJones2003}, the Taylor expansion of
\(\log\Phi(z)\) at \(z=0\) is
\begin{equation}
\label{eq:log-Phi-cumulant-series}
\sum_{n=1}^\infty  \frac{(-z)^n}{n!} \int_{W^n} \kappa_{Y_r}^{(n)}(x_1,\dots,x_n) \prod_{i=1}^n h(x_i)\,  dx_1\cdots  dx_n.
\end{equation}

We next show that this series converges in a disc containing \(z=1\).
Let
\(
 V_r=v_d(2r)^d.
\)
By Lemma~\ref{lem:integrated-cumulant-bound}, for \(n\ge2\),
\[
\frac{|z|^n}{n!}
\int_{W^n} \left|\kappa_{Y_r}^{(n)}(x_1,\dots,x_n)\right|
\prod_{i=1}^n h(x_i)\,  dx_1\cdots  dx_n
\le |W|\|h\|_\infty \frac{|z|^n}{n!} \lambda^n2^{n-1}V_r^{n-1}n^{n-2},
\]
where we used \(\|h\|_\infty^n\le\|h\|_\infty\).

Since
\(
 n!\ge (n/e)^n,
\)
we have
\(
\frac{n^{n-2}}{n!}\le\frac{e^n}{n^2}.
\)
It follows that the series in \eqref{eq:log-Phi-cumulant-series}
converges absolutely whenever
\(
2e|z|\lambda V_r<1.
\)
Choose \(c_d>0\) sufficiently small that
\(
\lambda r^d\le c_d \Longrightarrow 2e\lambda V_r\le\frac12.
\)
For example, one may take
\[
 c_d=\frac{1}{2^{d+2}ev_d}.
\]
The series in \eqref{eq:log-Phi-cumulant-series} then defines an analytic
function on a disc of radius strictly larger than \(1\). Near \(z=0\)
its exponential equals \(\Phi(z)\); hence, by analytic continuation,
the same identity holds throughout this disc. Evaluating at \(z=1\)
gives
\begin{equation}
\label{eq:Laplace-cumulant-expansion}
\log L_{Y_r}(g)
= \sum_{n=1}^\infty \frac{(-1)^n}{n!} \int_{W^n} \kappa_{Y_r}^{(n)}(x_1,\dots,x_n)
\prod_{i=1}^n h(x_i)\,   dx_1\cdots  dx_n.
\end{equation}
For \(z\in[0,1]\), \(\Phi(z)>0\), and the series is real-valued and
vanishes at \(z=0\); hence its value at \(z=1\) is the usual real
logarithm of \(L_{Y_r}(g)\).

The first factorial cumulant density of \(Y_r\) is
\( \kappa_{Y_r}^{(1)}(x)=\lambda_r'.
\)
On the other hand,
\[
\log L_{\Pi_r}(g)
= -\lambda_r'\int_W h(x)\,  dx.\]
Thus the first-order terms cancel, and
\[
\log L_{Y_r}(g)-\log L_{\Pi_r}(g)
=\sum_{n=2}^\infty\frac{(-1)^n}{n!} \int_{W^n}\kappa_{Y_r}^{(n)}(x_1,\dots,x_n)\prod_{i=1}^n h(x_i)\,  dx_1\cdots  dx_n.
\]

Applying Lemma~\ref{lem:integrated-cumulant-bound} once more, we obtain
\[
\begin{aligned}
\left| \log L_{Y_r}(g)-\log L_{\Pi_r}(g)\right|
&\le  2|W|\lambda^2V_r\|h\|_\infty\sum_{n=2}^\infty \frac{n^{n-2}}{n!}(2\lambda V_r)^{n-2}
\le 2e^2|W|\lambda^2V_r\|h\|_\infty \sum_{m=0}^\infty (2e\lambda V_r)^m.
\end{aligned}
\]
Under the above choice of \(c_d\), the last geometric series is bounded
by \(2\). Hence
\[
\left| \log L_{Y_r}(g)-\log L_{\Pi_r}(g) \right|
\le 4e^2|W|\lambda^2V_r\|h\|_\infty.
\]
Since
\(
 V_r=v_d(2r)^d, \,\|h\|_\infty\le\|g\|_\infty,
\)
the claimed estimate follows, for example with
\(
 C_d=4e^2v_d.
\)
\end{proof}

%========================================

\subsection{A Stein bound in terms of the short-range correlation}
\label{subsec:stein-bound}

We next compare \(\mathcal L(Y_{r,W})\) with
\(\mathcal L(\Pi_{r,W})\) using the Barbour--Brown Stein method in the
\(d_2\)-metric; see
\cite{BarbourBrown1992,Schuhmacher2009,Schuhmacher2009Bernoulli}.
The first result gives an explicit upper bound on \(d_2\) in terms of the
short-range deviation of the pair-correlation function \(g_{Y_r}\) from
\(1\).  The second result gives a converse lower bound detected by the
variance of the total number of retained points.

The proof of the next result applies Schuhmacher's local-dependence theorem with neighbourhoods \(A_x=W\cap B(x,2r)\). A Palm coupling is constructed so that the original and Palm versions coincide outside \(A_x\), after which the remaining local terms are expressed through the pair-correlation function.
\begin{theorem}[Stein--Poisson approximation]
\label{thm:stein-main}
Assume that \(\lambda_r'>0\). Then
\begin{equation}\label{eq:stein-bound}
\begin{aligned}
d_2\left(
\mathcal L(Y_{r,W}),
\mathcal L(\Pi_{r,W})
\right)
&\le \lambda_r'^2|W|\Bigg[\int_{\|h\|\le2r}\left|g_{Y_r}^{\circ}(\|h\|)-1\right|\,dh
+ 2v_d(2r)^d\Bigg].
\end{aligned}
\end{equation}
\end{theorem}

\begin{proof}

If \(|W|=0\), then both \(Y_{r,W}\) and \(\Pi_{r,W}\) are almost
surely empty, and the assertion is immediate. Henceforth assume that
\(|W|>0\).
To apply Theorem~3.A of Schuhmacher
\cite{Schuhmacher2009Bernoulli}, we follow notation of the paper and write
\(
\Xi:=Y_{r,W},
\)
and let
\(
\alpha(dx):=\lambda_r'\mathbf 1_W(x)\,dx
\)
be its intensity measure. Thus
\(
\alpha(W)=\lambda_r'|W|
\)
and \(\mathcal L(\Pi_{r,W})\) is the law of a Poisson process with intensity
measure \(\alpha\).

We first recall the metric appearing in the Palm-coupling term of
Schuhmacher's theorem. If
\[
\xi=\sum_{i=1}^n\delta_{x_i},
\qquad
\zeta=\sum_{j=1}^m\delta_{y_j}
\]
are finite counting measures on \(W\), define
\[
d_1'(\xi,\zeta)
:= \begin{cases} \displaystyle
(m-n)+
\min_{\pi\in S_m}
\sum_{i=1}^n d_0(x_i,y_{\pi(i)}),
& n\le m,\\[2ex]
d_1'(\zeta,\xi),
& m<n,
\end{cases}
\]
where \(S_m\) is the set of all permutations of
\(\{1,\dots,m\}\). 
For \(x\in W\), define
\[
A_x:=W\cap B(x,2r).
\]
Since \(x\in A_x\) and
\(
\{(x,y)\in W^2:y\in A_x\}
= \{(x,y)\in W^2:\|x-y\|\le2r\}
\)
is Borel, \((A_x)_{x\in W}\) is a measurable neighbourhood system.

Taking the \(\varepsilon_2\)-alternative in
Theorem~3.A of Schuhmacher
\cite[Theorem~3.A]{Schuhmacher2009Bernoulli}, we obtain
\begin{equation}\label{eq:Schuhmacher-used}
\begin{aligned}
d_2\bigl(\mathcal L(\Xi),\mathcal L(\Pi_{r,W})\bigr)
&\le B_2\bigl(\alpha(W)\bigr)
\Bigg[\int_W\alpha(A_x)\,\alpha(dx) +\mathbb E\int_W \bigl(\Xi(A_x)-1\bigr)\,\Xi(dx)\Bigg] \\
&\quad+B_2\bigl(\alpha(W)\bigr)\mathbb E\int_W d_1'\bigl( \Xi_{W\setminus A_x},(\Xi^x)_{W\setminus A_x} \bigr)\,\alpha(dx),
\end{aligned}
\end{equation}
where \(\Xi^x\) is a Palm version of \(\Xi\) at \(x\), and
\[
B_2(a)
:= \min\left\{
1,\, \frac{11}{6a} \left(1+2\log^+\frac{6a}{11}\right)
\right\},\qquad \log^+u:=\max\{\log u,0\}.
\]
In particular,
\(
0<B_2(a)\le1.
\)

We next show that the term involving \(d_1'\) in
\eqref{eq:Schuhmacher-used} vanishes. Let
\(\mathbb P^x\) be the one-point Palm distribution of \(\Xi\) at \(x\)
with respect to its intensity measure \(\alpha\), and let
\(\mathbb E^x\) denote the corresponding expectation. As in the Palm
notation introduced above, inside \(\mathbb E^x\) the symbol \(\Xi\)
denotes the canonical configuration distributed according to
\(\mathbb P^x\).

For \(x\in W\), define
\[
R_x:\mathbf N(W)\longrightarrow\mathbf N(W),
\qquad  R_x(\xi):=\xi_{W\setminus A_x}.
\]
Since
\(
\{(x,y)\in W^2:y\in A_x\}
=\{(x,y)\in W^2:\|x-y\|\le2r\}
\)
is Borel, the map
\(
(x,\xi)\longmapsto R_x(\xi)
\)
is measurable.

Let \(F:\mathbf N(W)\to[0,\infty)\) be bounded and measurable. By the
marked Mecke formula and the definition of the Palm distribution of
\(\Xi=Y_{r,W}\),
\begin{equation}\label{eq:Palm-outside-identity}
\mathbb E^x\left[
F\bigl(R_x(\Xi)\bigr)
\right]
= \frac{1}{m_p(\mu_r)} \mathbb E\left[ \eta_r(x;X+\delta_x) F\bigl(R_x(\Xi)\bigr)\right]
\end{equation}
for \(\alpha\)-almost every \(x\in W\). Here the configuration in the
right-hand side is the original thinned process \(\Xi=Y_{r,W}\).
Indeed, inserting the point at \(x\) does not alter any retention
decision at a point \(y\in W\setminus A_x\), because
\(\|x-y\|>2r\).

Moreover,
\(
\eta_r(x;X+\delta_x)
= p\left(X\bigl(B(x,r)\setminus\{x\}\bigr)\right)
\)
is measurable with respect to the restriction of the marked Poisson
input to \(B(x,r)\times[0,1]\). On the other hand,
\(R_x(\Xi)\) is measurable with respect to the restriction of the
marked Poisson input to
\(
B(x,r)^{\mathrm c}\times[0,1].
\)
Indeed, if \(y\in W\setminus A_x\), then \(\|x-y\|>2r\), and hence
\(
B(y,r)\subset B(x,r)^{\mathrm c}.
\)
The restrictions of the marked Poisson process to
\(B(x,r)\times[0,1]\) and
\(B(x,r)^{\mathrm c}\times[0,1]\) are independent. Consequently,
\[
\eta_r(x;X+\delta_x) \quad\text{and}\quad R_x(\Xi)
\]
are independent. Since
\(\mathbb E[\eta_r(x;X+\delta_x)]= m_p(\mu_r),
\)
equation \eqref{eq:Palm-outside-identity} gives
\(
\mathbb E^x\left[F\bigl(R_x(\Xi)\bigr)\right]
= \mathbb E\left[F\bigl(R_x(\Xi)\bigr)\right].
\)
Therefore, the restriction of the Palm configuration to \(W\setminus A_x\) has the same distribution as the corresponding restriction of the original process. Equivalently,
\begin{equation}\label{eq:outside-restriction-laws}
        \mathbb P^x\circ R_x^{-1}
        =\mathcal L(\Xi)\circ R_x^{-1}
\end{equation}
for \(\alpha\)-almost every \(x\in W\).

We now construct the required coupling measurably in \(x\). Put
\[\nu_x:= \mathbb P^x\circ R_x^{-1} =\mathcal L(\Xi)\circ R_x^{-1}.
\]
Since \(W\) is a compact metric space, \(\mathbf N(W)\) is a  Borel
space. Hence there exist probability kernels
\[
 K_x^{0}(\zeta,d\xi), \qquad  K_x^{1}(\zeta,d\xi'),
\]
jointly measurable in \((x,\zeta)\), which give regular conditional
distributions of \(\mathcal L(\Xi)\) and \(\mathbb P^x\), respectively,
given the value of \(R_x\). Thus
\[
\mathcal L(\Xi)(d\xi)
=\int_{\mathbf N(W)}K_x^{0}(\zeta,d\xi)\,\nu_x(d\zeta) \quad \text{and}\quad
\mathbb P^x(d\xi')
= \int_{\mathbf N(W)} K_x^{1}(\zeta,d\xi')\,\nu_x(d\zeta),
\]
where, for \(\nu_x\)-almost every \(\zeta\),
\[
K_x^{0}\bigl(
\zeta,\{\xi:R_x(\xi)=\zeta\} \bigr)=1
\quad \text{and}\quad
K_x^{1}\bigl( \zeta,\{\xi':R_x(\xi')=\zeta\}\bigr)=1.
\]

Define the probability kernel
\[
Q_x(d\xi,d\xi')
:= \int_{\mathbf N(W)} K_x^{0}(\zeta,d\xi) K_x^{1}(\zeta,d\xi')\,\nu_x(d\zeta).
\]
Its first marginal is \(\mathcal L(\Xi)\), its second marginal is
\(\mathbb P^x\), and
\[
Q_x\left(\left\{(\xi,\xi'):R_x(\xi)=R_x(\xi')\right\}\right)=1
\]
for \(\alpha\)-almost every \(x\).

Let \(\Xi^x\) be a Palm version of \(\Xi\) at \(x\), jointly coupled
with \(\Xi\) according to \(Q_x\). Then
\(
\mathcal L(\Xi^x)=\mathbb P^x
\)
and
\(
\Xi_{W\setminus A_x}=(\Xi^x)_{W\setminus A_x}\)
almost surely
for \(\alpha\)-almost every \(x\). It follows that
\[
d_1'\left(\Xi_{W\setminus A_x},(\Xi^x)_{W\setminus A_x}\right)=0
\]
almost surely for \(\alpha\)-almost every \(x\), and therefore
\begin{equation}\label{eq:Palm-term-zero}
\mathbb E\int_W
d_1'\left(
\Xi_{W\setminus A_x},
(\Xi^x)_{W\setminus A_x}
\right)\,\alpha(dx)=0.
\end{equation}

Since \(B_2(\alpha(W))\le1\), equations
\eqref{eq:Schuhmacher-used} and \eqref{eq:Palm-term-zero} yield
\begin{equation}\label{eq:stein-start}
d_2\bigl(\mathcal L(\Xi),\mathcal L(\Pi_{r,W})\bigr)
\le I+II,
\end{equation}
where
\[
I:=\int_W\alpha(A_x)\,\alpha(dx), \qquad
II:=\mathbb E\int_W\bigl(\Xi(A_x)-1\bigr)\,\Xi(dx).
\]

For the first term,
\[
I=\lambda_r'^2\int_W |W\cap B(x,2r)|\,dx \le \lambda_r'^2|W|\,v_d(2r)^d.
\]

We now evaluate \(II\). Since \(x\in A_x\), for \(\Xi\)-almost every
\(x\),
\(
\Xi(A_x)-1=(\Xi-\delta_x)(A_x).
\)
Therefore, using the second factorial measure of \(\Xi\),
\[
II= \mathbb E \int_{W\times W} \mathbf 1\{y\in A_x\}\,\Xi^{(2)}(d(x,y))=
\int_{W\times W} \mathbf 1\{y\in A_x\}\, \rho_\Xi^{(2)}(x,y)\,dx\,dy.
\]
Since \(\Xi=(Y_r)_W\), its second-order product density on \(W^2\) is
\[
\rho_\Xi^{(2)}(x,y)
= \lambda_r'^2 g_{Y_r}^{\circ}(\|y-x\|).
\]
Thus
\[
II = \lambda_r'^2 \int_{W\times W} \mathbf 1\{\|y-x\|\le2r\} g_{Y_r}^{\circ}(\|y-x\|) \,dx\,dy.
\]

Combining the expressions for \(I\) and \(II\), we obtain
\[
\begin{aligned}
I+II &=
\lambda_r'^2\int_{W\times W}\mathbf 1\{\|y-x\|\le2r\}\left[
1+g_{Y_r}^{\circ}(\|y-x\|)\right]
\,dx\,dy \\
&\le \lambda_r'^2 \int_{W\times W} \mathbf 1\{\|y-x\|\le2r\}
\left[ 2+ \left|g_{Y_r}^{\circ}(\|y-x\|)-1\right| \right]\,dx\,dy,
\end{aligned}
\]
where we used
\(
1+g=2+(g-1)\le2+|g-1|.
\)

For any nonnegative measurable function \(\varphi\),
\[
\int_{W\times W}\varphi(y-x)\,dx\,dy
=\int_{\mathbb R^d} \varphi(h)\,|W\cap(W-h)|\,dh.
\]
Since
\(
|W\cap(W-h)|\le|W|,
\)
we conclude that
\[
I+II\le \lambda_r'^2|W| \int_{\|h\|\le2r}
\left| g_{Y_r}^{\circ}(\|h\|)-1 \right|\,dh 
+ 2\lambda_r'^2|W|\,v_d(2r)^d.
\]
Together with \eqref{eq:stein-start}, this implies \eqref{eq:stein-bound}.
\end{proof}

The preceding theorem gives an upper bound in terms of the short-range pair correlation. We now obtain a converse estimate from the variance of the total number of retained points. Any \([0,1]\)-valued function of the total count is an admissible \(d_2\)-test function, because configurations of different cardinalities have \(d_1\)-distance one. The proof applies this observation to a truncated centred-square function and optimises the truncation level, thereby converting the count-variance defect into a quantitative lower bound.

\begin{theorem}[Count-variance lower bound]
\label{prop:d2-lower}
Assume that \(\lambda_r'>0\), and put
\[
N_{r,W}:=Y_r(W), \quad\bar N_{r,W}:=\mathbb E N_{r,W},
\]
\[
\Delta_{r,W} := \operatorname{Var}(N_{r,W})-\bar N_{r,W},
\quad  M_{r,W} :=  \mathbb E\bigl[(N_{r,W}-\bar N_{r,W})^4\bigr]+\bar N_{r,W} +3\bar N_{r,W}^2.
\]
Then 
\begin{equation} \label{eq:DeltaW-pair-correlation} \Delta_{r,W} = \lambda_r'^2 \int_{\|h\|\le2r} \left( g_{Y_r}^{\circ}(\|h\|)-1 \right) |W\cap(W-h)|\, d h. \end{equation} 
Moreover, if \(\Delta_{r,W}\ne0\), then 
\begin{equation} 
\label{eq:d2-count-variance-optimized}
d_2\left( \mathcal L(Y_{r,W}), \mathcal L(\Pi_{r,W}) \right) \ge \frac{\Delta_{r,W}^2}{4M_{r,W}}. \end{equation}
\end{theorem}

\begin{proof}
Since \(Y_r\) is stationary with intensity \(\lambda_r'\),
\(
\mathbb E N_{r,W}=\lambda_r'|W|=\bar N_{r,W}.
\)
By the second factorial moment formula and the definition of the
pair-correlation function,
\[
\mathbb E[N_{r,W}(N_{r,W}-1)]
= \int_{W\times W}\rho_{Y_r}^{(2)}(x,y)\,  dx\,  dy
= \lambda_r'^2\int_{W\times W}g_{Y_r}^{\circ}(\|x-y\|)\,  dx\,  dy.
\]
Consequently,
\[
\Delta_{r,W}=
\operatorname{Var}(N_{r,W})-\bar N_{r,W}
=\mathbb E[N_{r,W}(N_{r,W}-1)]-(\mathbb E N_{r,W})^2
=\lambda_r'^2\int_{W\times W}\left(g_{Y_r}^{\circ}(\|x-y\|)-1\right)
\,  dx\,  dy.
\]
Using
\[
\int_{W\times W}f(x-y)\,  dx\,  dy
=\int_{\R^d}f(h)|W\cap(W-h)|\,  dh
\]
and the fact that
\(
g_{Y_r}^{\circ}(\|h\|)=1,\,\|h\|>2r,
\)
we obtain \eqref{eq:DeltaW-pair-correlation}.

It remains to prove the lower bound. Let
\(F:\mathbb N_0\to[0,1]\), and define
\[
  H_F(\xi):=F(\xi(W)),\qquad \xi\in\mathbf N(W).
\]
If \(\xi(W)=\zeta(W)\), then
\(H_F(\xi)=H_F(\zeta)\). If their cardinalities differ, then
\(d_1(\xi,\zeta)=1\), while
\[
|H_F(\xi)-H_F(\zeta)|\le1.
\]
Thus \(H_F\) is admissible in the definition of \(d_2\). Writing
\(
 N_\Pi:=\Pi_{r,W}(W)\sim\operatorname{Poisson}(\bar N_{r,W}),
\)
we obtain
\begin{equation}
\label{eq:d2-count-test}
d_2\left(\mathcal L(Y_{r,W}),\mathcal L(\Pi_{r,W})\right)
\ge
\left|\mathbb EF(N_{r,W})-\mathbb EF(N_\Pi)\right|.
\end{equation}

For \(K>0\), let
\[
 F_K(n) := \frac{(n-\bar N_{r,W})^2\wedge K}{K},\qquad n\in\mathbb N_0.
\]
Applying \eqref{eq:d2-count-test} with \(F=F_K\) gives
\[
d_2\left(\mathcal L(Y_{r,W}),\mathcal L(\Pi_{r,W})\right)
\ge \frac1K \Big| \mathbb E\bigl[(N_{r,W}-\bar N_{r,W})^2\wedge K\bigr]
- \mathbb E\bigl[(N_\Pi-\bar N_{r,W})^2\wedge K\bigr] \Big|.
\]

For \(u\in\R\), put
\(
R_K(u):=u^2-(u^2\wedge K).
\)
Then
\[
 0\le R_K(u)\le\frac{u^4}{K}.
\]
The reverse triangle inequality therefore yields
\[
\begin{aligned}
&
\Big| \mathbb E\bigl[(N_{r,W}-\bar N_{r,W})^2\wedge K\bigr] - \mathbb E\bigl[(N_\Pi-\bar N_{r,W})^2\wedge K\bigr]
\Big|\\
&\ge  \left| \mathbb E(N_{r,W}-\bar N_{r,W})^2
- \mathbb E(N_\Pi-\bar N_{r,W})^2 \right|
- \frac1K\mathbb E[(N_{r,W}-\bar N_{r,W})^4]
- \frac1K\mathbb E[(N_\Pi-\bar N_{r,W})^4].
\end{aligned}
\]
Since
\(
\mathbb E(N_{r,W}-\bar N_{r,W})^2
- \mathbb E(N_\Pi-\bar N_{r,W})^2 = \Delta_{r,W}
\)
and
\(
\mathbb E[(N_\Pi-\bar N_{r,W})^4] = \bar N_{r,W}+3\bar N_{r,W}^2,
\)
we conclude that
\[
d_2\left( \mathcal L(Y_{r,W}), \mathcal L(\Pi_{r,W}) \right)
\ge  \frac1K \left( |\Delta_{r,W}|-\frac{M_{r,W}}{K} \right).
\]
Choosing
\[
 K=\frac{2M_{r,W}}{|\Delta_{r,W}|}
\]
gives \eqref{eq:d2-count-variance-optimized}.
\end{proof}

%========================================
\subsection{Small-radius regime and explicit rate}
\label{subsec:stein-smallr}

We now specialise the Stein bound to the regime \(r\downarrow0\), with
\(\lambda\) and \(W\) fixed. We combine
Theorem~\ref{thm:stein-main} with the contact-scale expansions of the
radial pair-correlation function \(g_{Y_r}^{\circ}\) obtained in
Corollaries~\ref{lem:g-first-order} and
\ref{lem:g-first-order-p0zero}. Recall that
\(
\mu_r=\lambda v_dr^d
\)
and
\( g_{Y_r}^{\circ}(s)=1,\, s>2r.
\)
All asymptotic statements below are understood as \(r\downarrow0\),
with \(d,\lambda,p\), and \(W\) fixed.

Let \(S_{d-1}:=|\mathbb S^{d-1}|=dv_d,\) denote \((d-1)-\)dimensional surface measure of a unit sphere \(\mathbb S^{d-1}:=\{u\in\R^d:\|u\|=1\}\). By $V_r = v_d(2r)^d$ we denote the volume of \(d-\)dimensional ball with radius $2r.$

Denote
\[ M_d^{(\le1)} := \int_0^1 t^{d-1}\omega_d(t)\,d t, \qquad M_d^{(>1)} := \int_1^2 t^{d-1}\omega_d(t)\, d t. \]

\begin{theorem}[Small-radius \(d_2\) upper bounds]
\label{prop:stein-smallr}
Let \(W\subset\R^d\) be a compact Borel set. All asymptotic
statements below are understood as \(r\downarrow0\), with
\(d,\lambda,p\), and \(W\) fixed. Put
\[
 c_1:=\frac{p(1)}{p(0)}\quad (\text{if } p(0)>0),  \qquad c_2:=\frac{p(2)}{p(1)}.
\]

\emph{(a) Suppose that \(p(0)>0\) and \(p(1)>0\). Then}
\begin{equation}
\label{eq:smallr-d2-upper}
\begin{aligned}
d_2\left(\mathcal L(Y_{r,W}),\mathcal L(\Pi_{r,W})\right)
\le& \lambda_r'^2|W|S_{d-1}r^d\Bigg\{\frac{|c_1^2-1|}{d}+\frac{2^{d+1}}{d}\\
&\qquad+\mu_r\left[M_d^{(>1)}(1-c_1)^2+c_1^2M_d^{(\le1)}(1-c_2)^2+\frac{2c_1^2}{d}|c_2-c_1|\right]
\Bigg\}
+O(r^{3d}).
\end{aligned}
\end{equation}

\emph{(b) Suppose that \(p(0)=0<p(1)\). Then}
\begin{equation}
\label{eq:p0zero-d2}
d_2\left(\mathcal L(Y_{r,W}),\mathcal L(\Pi_{r,W})\right)
\le \lambda_r'^2|W|S_{d-1}r^d\Bigg\{\frac{1}{d\mu_r^2}
+\frac{1}{\mu_r}\left[\frac{c_2}{d}
+(1-c_2)^2M_d^{(\le1)}+M_d^{(>1)}\right]\Bigg\}
+O(r^{3d}).
\end{equation}
\end{theorem}

\begin{proof}
Using polar coordinates,
\begin{equation}
\label{eq:polar-smallr}
\int_{\|h\|\le2r}\left|g_{Y_r}^{\circ}(\|h\|)-1\right|\, d h
=S_{d-1}r^d\int_0^2t^{d-1}\left|g_{Y_r}^{\circ}(tr)-1\right|\, d t.
\end{equation}

Suppose first that \(p(0)>0\) and \(p(1)>0\). By
Corollary~\ref{lem:g-first-order}, uniformly for \(0<t\le1\),
\[
\begin{aligned}
\left|g_{Y_r}^{\circ}(tr)-1\right|
\le{}&|c_1^2-1|
+c_1^2\mu_r\omega_d(t)(1-c_2)^2
+2c_1^2\mu_r|c_2-c_1|
+O(\mu_r^2),
\end{aligned}
\]
while, uniformly for \(1<t\le2\),
\[
\left|g_{Y_r}^{\circ}(tr)-1\right|
\le \mu_r\omega_d(t)(1-c_1)^2
+ O(\mu_r^2).
\]
Substitution into \eqref{eq:polar-smallr} gives
\begin{equation}
\label{eq:pair-integral-generic}
\begin{aligned}
&\int_{\|h\|\le2r}\left|g_{Y_r}^{\circ}(\|h\|)-1\right|\, d h\\
&\quad\le S_{d-1}r^d \Bigg\{ \frac{|c_1^2-1|}{d}
+\mu_r\left[M_d^{(>1)}(1-c_1)^2
+c_1^2M_d^{(\le1)}(1-c_2)^2
+\frac{2c_1^2}{d}|c_2-c_1|\right]\Bigg\}
+O(\mu_r^2r^d).
\end{aligned}
\end{equation}

Applying Theorem~\ref{thm:stein-main}, we obtain
\[
\begin{aligned}
d_2\left(\mathcal L(Y_{r,W}),\mathcal L(\Pi_{r,W})\right)
&\le\lambda_r'^2|W|\Bigg[S_{d-1}r^d\Bigg\{\frac{|c_1^2-1|}{d}\\
&+\mu_r\left[M_d^{(>1)}(1-c_1)^2
+c_1^2M_d^{(\le1)}(1-c_2)^2
+\frac{2c_1^2}{d}|c_2-c_1|\right]\Bigg\}
+2V_r\Bigg]
+O\left(\lambda_r'^2|W|\mu_r^2r^d\right).
\end{aligned}
\]
Since
\(\lambda_r'= \lambda p(0)+O(\mu_r),
\)
the sequence \((\lambda_r')_{r>0}\) remains bounded as \(r\downarrow0\).
Moreover,
\(\mu_r^2r^d=O(r^{3d}).
\)
This proves \eqref{eq:smallr-d2-upper}.

Suppose next that \(p(0)=0<p(1)\). By
Corollary~\ref{lem:g-first-order-p0zero}, uniformly for \(0<t\le1\),
\[
g_{Y_r}^{\circ}(tr)
=\frac{1}{\mu_r^2}\left\{1+\mu_r\left[
c_2+\omega_d(t)(1-c_2)^2\right]\right\}
+O(1),
\]
while, uniformly for \(1<t\le2\),
\[
g_{Y_r}^{\circ}(tr)
=\frac{\omega_d(t)}{\mu_r}+O(1).
\]
Consequently,
\[
\left|g_{Y_r}^{\circ}(tr)-1\right|
\le \frac{1}{\mu_r^2}+\frac{c_2+\omega_d(t)(1-c_2)^2}{\mu_r}
+O(1),\qquad 0<t\le1,
\]
and
\[
\left|g_{Y_r}^{\circ}(tr)-1\right|
\le\frac{\omega_d(t)}{\mu_r}+O(1),\qquad 1<t\le2,
\]
where both \(O(1)\)-terms are uniform in \(t\).

Using \eqref{eq:polar-smallr}, we obtain
\begin{equation}
\label{eq:p0zero-pair-integral}
\int_{\|h\|\le2r}\left|g_{Y_r}^{\circ}(\|h\|)-1\right|\, d h
\le S_{d-1}r^d \Bigg\{ \frac{1}{d\mu_r^2}
+ \frac{1}{\mu_r} \left[ \frac{c_2}{d}
+ (1-c_2)^2M_d^{(\le1)}
+ M_d^{(>1)} \right] \Bigg\}
+ O(r^d).
\end{equation}

Theorem~\ref{thm:stein-main} therefore yields
\begin{align*}
d_2\left(
\mathcal L(Y_{r,W}),
\mathcal L(\Pi_{r,W})
\right)
\le{}&
\lambda_r'^2|W|S_{d-1}r^d
\Bigg\{ \frac{1}{d\mu_r^2}
+ \frac{1}{\mu_r}\left[ \frac{c_2}{d}
+(1-c_2)^2M_d^{(\le1)}
+ M_d^{(>1)}\right]\Bigg\}\\
&+ O\left(\lambda_r'^2|W|r^d\right)
+2\lambda_r'^2|W|V_r.
\end{align*}

Since \(p(0)=0\),
\(
m_p(\mu_r)
=p(1)\mu_r+O(\mu_r^2),
\)
and hence
\(
\lambda_r'= \lambda p(1)\mu_r+O(\mu_r^2) = O(r^d).
\)
It follows that
\(
 \lambda_r'^2r^d=O(r^{3d})
\)
and, because \(V_r=O(r^d)\),
\(
\lambda_r'^2V_r=O(r^{3d}).
\)
Thus, both of the last two terms may be absorbed into \(O(r^{3d})\),
which proves \eqref{eq:p0zero-d2}.
\end{proof}

\begin{remark}[On the leading terms]
\label{rem:small-r-leading-terms}
\emph{(i)} Suppose that \(p(0),p(1)>0\). The contact contribution
\( \frac 1d{|c_1^2-1|}\,r^d\)
vanishes if and only if \(p(1)=p(0)\). The terms involving
\(c_2\), \(M_d^{(\le1)}\), and \(M_d^{(>1)}\) are multiplied by
\(\mu_r r^d\) and are therefore of order \(r^{2d}\) for fixed
\(\lambda\). However, the neighbourhood-volume term
\( \frac 1d{2^{d+1}}\,r^d\)
is present independently of \(p\). Consequently,
Theorem~\ref{prop:stein-smallr}\emph{(a)} yields an \(O(r^d)\) upper
bound even when \(p(1)=p(0)\), although in that case the
pair-correlation contribution improves to order \(r^{2d}\).

\emph{(ii)} Suppose that \(p(0)=0<p(1)\). Since
\(\lambda_r'  =  \lambda p(1)\mu_r+O(\mu_r^2),
\)
the term
\( \lambda_r'^2 r^d\mu_r^{-2}
\)
is of order \(\lambda^2r^d\), whereas the terms multiplied by
\(\mu_r^{-1}\) are of order \(r^{2d}\) for fixed \(\lambda\).
Thus the dominant term in
Theorem~\ref{prop:stein-smallr}\emph{(b)} comes from the
mutual-inclusion region \(0<t\le1\) and gives
\(
d_2\left( \mathcal L(Y_{r,W}), \mathcal L(\Pi_{r,W})\right) = O\left( r^d\right).
\)
\end{remark}

The next theorem shows that the order \(r^d\) is sharp in the generic fixed-intensity case. The count-variance lower bound of Theorem~\ref{prop:d2-lower} is quadratic in the variance defect and therefore does not by itself recover this order. The proof first identifies the leading \(r^d\)-term of the variance defect and then detects it linearly by means of a bounded probability generating-function statistic of the total count. The higher-order factorial cumulants contribute only \(O(r^{2d})\), while the matching upper bound follows from Theorem~\ref{prop:stein-smallr}.

\begin{theorem}[Fixed-intensity small-radius sharpness]
\label{thm:fixed-lambda-small-r-sharpness}
Let \(W\subset\R^d\) be a compact Borel set with \(|W|>0\), and assume
\[
 p(0)>0, \qquad    p(1)>0,  \qquad  p(1)\ne p(0).
\]
Then there exist constants \(c,C\in(0,\infty)\) and \(r_0>0\) such
that, for every \(0<r<r_0\),
\[
cr^d\le  d_2\left(\mathcal L(Y_{r,W}), \mathcal L(\Pi_{r,W}) \right)\le Cr^d.
\]
\end{theorem}

\begin{proof}
For each \(r>0\), we use the notation
\(N_{r,W}:=Y_{r,W}(W),\, \bar N_{r,W}:=\lambda_r'|W|, \,
\Delta_{r,W}:=\operatorname{Var}(N_{r,W})-\bar N_{r,W}
\)
from Theorem~\ref{prop:d2-lower}. Since \(p(0)>0\), one has
\(\lambda_r'>0\).

We first determine the small-radius asymptotic behaviour of the variance
defect \(\Delta_{r,W}\). By \eqref{eq:DeltaW-pair-correlation},
\begin{equation}
\label{eq:Delta-smallr-start}
\Delta_{r,W}
=\lambda_r'^2\int_{\|h\|\le2r}\left(g_{Y_r}^{\circ}(\|h\|)-1\right)
|W\cap(W-h)|\, d h.
\end{equation}

Put
\[
  a_p := \left(\frac{p(1)}{p(0)}\right)^2-1.
\]
By Corollary~\ref{lem:g-first-order}, uniformly for
\(0<\|h\|\le r\),
\(
g_{Y_r}^{\circ}(\|h\|)-1 =a_p+O(\mu_r),
\)
whereas, uniformly for \(r<\|h\|\le2r\),
\(
 g_{Y_r}^{\circ}(\|h\|)-1=O(\mu_r).
\)

Since \(\mathbf 1_W\in L^1(\R^d)\), continuity of translations in
\(L^1(\R^d)\) implies
\[
|W\cap(W-h)|
=|W|-\frac12|W\triangle(W-h)|=|W|+o(1)
\]
uniformly for \(\|h\|\le r\) as \(r\downarrow0\). Consequently,
\begin{equation}
\label{eq:window-overlap-smallr}
\int_{\|h\|\le r}|W\cap(W-h)|\, d h
=|W|v_dr^d+o(r^d).
\end{equation}

Splitting the integral in \eqref{eq:Delta-smallr-start} into the regions
\(\{\|h\|\le r\}\) and \(\{r<\|h\|\le2r\}\), and using
\(
|W\cap(W-h)|\le |W|,
\)
we obtain
\[
\int_{\|h\|\le2r}
\left(g_{Y_r}^{\circ}(\|h\|)-1\right)|W\cap(W-h)|\, d h
=a_p\int_{\|h\|\le r}|W\cap(W-h)|\, d h
+O(\mu_rr^d)
=a_p|W|v_dr^d+o(r^d),
\]
where we used \(\mu_r=O(r^d)\).

Moreover,
\(
m_p(\mu_r)=p(0)+O(\mu_r),
\)
and therefore
\(
\lambda_r' =\lambda p(0)+O(r^d).
\)
Substituting these estimates into
\eqref{eq:Delta-smallr-start}, we obtain
\begin{equation}
\label{eq:Delta-smallr-asymptotic}
\Delta_{r,W}= D_pr^d+o(r^d),
\end{equation}
where
\(
D_p:= \lambda^2 \bigl(p(1)^2-p(0)^2\bigr) |W|v_d.
\)
Since \(p(1)\ne p(0)\), we have \(D_p\ne0\).

The lower bound
\eqref{eq:d2-count-variance-optimized} is quadratic in \(\Delta_{r,W}\) and
would therefore yield only an estimate of order \(r^{2d}\). To detect
the variance defect linearly, we use a different bounded function of
the total count.

Fix \(\tau\in(-1,0)\). For all sufficiently small \(r\), the condition
\(\lambda r^d\le c_d\) in Theorem~\ref{thm:laplace-control} is
satisfied. Apply the factorial-cumulant expansion
\eqref{eq:Laplace-cumulant-expansion} with
\[
 g=-\log(1+\tau)\,\mathbf 1_W.
\]
Integrating the function $g$ with respect to the point process $Y_r$
\[
\int_{\R^d}g(x)\,Y_r(dx)
=-\log(1+\tau)\,Y_r(W)
=-\log(1+\tau)\,N_{r,W},
\]
we get the Laplace functional:
\[
\exp\left(-\int_{\R^d}g(x)\,Y_r(dx)\right)
=(1+\tau)^{N_{r,W}}.
\]
Using
\(1-e^{-g(x)}=-\tau\mathbf 1_W(x),
\)
 \eqref{eq:Laplace-cumulant-expansion} becomes
\[
\log\mathbb E[(1+\tau)^{N_{r,W}}]
=\sum_{j=1}^{\infty}\frac{\tau^j}{j!}\int_{W^j}\kappa_{Y_r}^{(j)}(x_1,\dots,x_j)\, d x_1\cdots d x_j.
\]
Separating the terms \(j=1\) and \(j=2\) gives
\begin{equation} 
\label{eq:count-generating-expansion} \log\mathbb E[(1+\tau)^{N_{r,W}}] = \tau\bar N_{r,W} + \frac{\tau^2}{2}\Delta_{r,W} + R_r(\tau). 
\end{equation}
where
\[
R_r(\tau)
=\sum_{j=3}^{\infty}\frac{\tau^j}{j!}\int_{W^j}\kappa_{Y_r}^{(j)}(x_1,\dots,x_j)\, d x_1\cdots d x_j.
\]
Here we used 
\[
\int_W\kappa_{Y_r}^{(1)}(x)\, d x
=\bar N_{r,W}, \ \text{and }\int_{W^2}\kappa_{Y_r}^{(2)}(x_1,x_2)
\, d x_1\, d x_2
=\Delta_{r,W}.
\]

By Lemma~\ref{lem:integrated-cumulant-bound}, and using  \(j!\ge(j/e)^j\), we obtain
\[
|R_r(\tau)|
\le|W|\sum_{j=3}^{\infty}\frac{|\tau|^j}{j!}\lambda^j2^{j-1}V_r^{j-1}j^{j-2}\le
e|\tau|\lambda|W|\sum_{j=3}^{\infty}\frac{\bigl(2e|\tau|\lambda V_r\bigr)^{j-1}}{j^2}.
\]

Since \(V_r\to0\), the series on the right is \(O(V_r^2)\). Thus
\begin{equation}
\label{eq:R-smallr}
 R_r(\tau)=O(V_r^2)=O(r^{2d}).
\end{equation}

To turn the difference between the two count generating functions into a lower bound for \(d_2\), we use the following admissible test function:
\[
H_\tau(\xi):=(1+\tau)^{\xi(W)},
\qquad \xi\in\mathbf N(W).
\]
Since \(\tau\in(-1,0)\), one has \(0\le H_\tau\le1\). If two
configurations have the same cardinality, then \(H_\tau\) takes the
same value on them; if their cardinalities differ, their
\(d_1\)-distance is one. Hence \(H_\tau\) is an admissible test
function in the definition of \(d_2\).

Since
 \(N_\Pi:=\Pi_{r,W}(W)\sim\operatorname{Poisson}(\bar N_{r,W}),\)
\(
\mathbb E[(1+\tau)^{N_\Pi}] =  e^{\tau\bar N_{r,W}}.
\)
Consequently, by \eqref{eq:count-generating-expansion},
\[
d_2\left(\mathcal L(Y_{r,W}),\mathcal L(\Pi_{r,W})\right)
\ge\left|\mathbb E[(1+\tau)^{N_{r,W}}]
-\mathbb E[(1+\tau)^{N_\Pi}]\right|
=e^{\tau\bar N_{r,W}}\left|\exp\left\{\frac{\tau^2}{2}\Delta_{r,W}+R_r(\tau)
\right\}-1\right|.
\]

By \eqref{eq:Delta-smallr-asymptotic} and
\eqref{eq:R-smallr}, the exponent in the last expression is
\[
x_r :=\frac{\tau^2}{2}\Delta_{r,W}+R_r(\tau)=\frac{\tau^2D_p}{2}r^d+o(r^d).
\]
Since \(D_p\ne0\), there exists \(r_1>0\) such that
\[
 |x_r|\ge\frac{|\tau^2D_p|}{4}r^d, \qquad 0<r<r_1.
\]
Also \(x_r\to0\), and hence, after decreasing \(r_1\) if necessary,
\[
 |e^{x_r}-1|\ge\frac12|x_r|.
\]
Finally,
\[
 \bar N_{r,W}= \lambda_r'|W| \longrightarrow  \lambda p(0)|W|,
\]
so \(e^{\tau\bar N_{r,W}}\) remains bounded away from zero. Therefore there
exists \(c>0\) such that
\[
d_2\left(\mathcal L(Y_{r,W}),\mathcal L(\Pi_{r,W})\right)
\ge cr^d
\]
for all sufficiently small \(r\).

The upper bound follows from
Theorem~\ref{prop:stein-smallr}\emph{(a)}. Indeed,
\[
\lambda_r'=\lambda p(0)+O(r^d),
\qquad \mu_r=O(r^d), \qquad  V_r=O(r^d),
\]
so the right-hand side of
\eqref{eq:smallr-d2-upper} is \(O(r^d)\). Hence there exists \(C>0\)
such that
\[
d_2\left(\mathcal L(Y_{r,W}),\mathcal L(\Pi_{r,W})\right)
\le  Cr^d
\]
for all sufficiently small \(r\). This completes the proof.
\end{proof}

%========================================

\subsection{Moderate-radius regime}
\label{subsec:stein-moderate-r}

When \(\mu_r=\lambda v_dr^d\) is not small, the contact-scale
expansions used in Subsection~\ref{subsec:stein-smallr} need not provide
an accurate approximation. Nevertheless, the exact representation of
the pair-correlation function in Theorem~\ref{thm:g-exact-PPP},
together with the Poisson Poincar\'e inequality, yields a
non-asymptotic bound valid for every \(r>0\).

Define the discrete Lipschitz constant
\[
L_p:=\sup_{k\ge0}|p(k+1)-p(k)|
\]
and the shifted retention mean
\[
m_p^+(\mu_r):=
\mathbb E[p(N_r+1)],
\qquad N_r\sim\operatorname{Poisson}(\mu_r).
\]
Recall that
\(
m_p(\mu_r)=\mathbb E[p(N_r)],
 \, \lambda_r'=\lambda m_p(\mu_r).
\)
For \(t\in(0,2]\), put
\(
I_t:=\mathbf 1\{t\le1\}.
\)
Let
\[
 U_t\sim \operatorname{Poisson}\bigl(\mu_r\omega_d(t)\bigr),
\qquad Z_t\sim \operatorname{Poisson} \bigl(\mu_r(1-\omega_d(t))\bigr)
\]
be independent, and define
\[
\psi_t(k)
:= \mathbb E\bigl[p(k+I_t+Z_t)\bigr],
\qquad k\in\mathbb N_0.
\]

The next decomposition separates two sources of deviation from Poisson behaviour: the dependence created by the common overlap count and, for \(t\le1\), the shift caused by the two distinguished points being mutual neighbours. It follows by conditioning on the common Poisson count \(U_t\).
\begin{lemma}[Covariance decomposition of the pair correlation]
\label{lem:cov-decomp}
Assume that \(\lambda_r'>0\). Then, for every \(t\in(0,2]\),
\begin{equation}
\label{eq:moderate-cov-decomp}
g_{Y_r}^{\circ}(tr)-1
=\frac{\operatorname{Var}(\psi_t(U_t))} {m_p(\mu_r)^2}
+ I_t\left(\frac{m_p^+(\mu_r)^2}{m_p(\mu_r)^2}-1\right).
\end{equation}
\end{lemma}

\begin{proof}
By Theorem~\ref{thm:g-exact-PPP}, the two neighbour counts entering the
joint retention probability may be represented as
\[
 U_t+I_t+Z_t^{(1)},
 \qquad
 U_t+I_t+Z_t^{(2)},
\]
where \(Z_t^{(1)}\) and \(Z_t^{(2)}\) are independent copies of \(Z_t\),
independent of \(U_t\). Conditioning on \(U_t\), we obtain
\[
\mathbb E\left[p(U_t+I_t+Z_t^{(1)})p(U_t+I_t+Z_t^{(2)})\,\middle|\,U_t\right]
=\psi_t(U_t)^2.
\]
Hence
\[
g_{Y_r}^{\circ}(tr)
=\frac{\mathbb E[\psi_t(U_t)^2]}{m_p(\mu_r)^2}.
\]

Since \(U_t+Z_t\sim\operatorname{Poisson}(\mu_r)\),
\[
\mathbb E[\psi_t(U_t)] =
\mathbb E[p(N_r+I_t)]
=\begin{cases}
m_p^+(\mu_r),&0<t\le1,\\
m_p(\mu_r),&1<t\le2.
\end{cases}
\]
Using
\(
\mathbb E[\psi_t(U_t)^2]
=\operatorname{Var}(\psi_t(U_t))
+\mathbb E[\psi_t(U_t)]^2
\)
gives \eqref{eq:moderate-cov-decomp}.
\end{proof}

We now bound the variance term in the preceding decomposition using the Poisson Poincaré inequality. The discrete Lipschitz constant \(L_p\) controls the add-one difference of \(\psi_t\), resulting in a pointwise estimate for \(\lvert g_{Y_r}^{\circ}(tr)-1\rvert\).
\begin{lemma}[Pointwise pair-correlation bound]
\label{lem:pointwise-gminus1}
Assume that \(\lambda_r'>0\). Then, for every \(t\in(0,2]\),
\begin{equation}
\label{eq:pointwise-gminus1}
\left| g_{Y_r}^{\circ}(tr)-1\right|
\le \frac{\mu_r\omega_d(t)L_p^2}{m_p(\mu_r)^2}
+I_t\left|\frac{m_p^+(\mu_r)^2}{m_p(\mu_r)^2}-1\right|.
\end{equation}
\end{lemma}

\begin{proof}
The Poisson Poincar\'e inequality gives
\[
\operatorname{Var}(\psi_t(U_t))
\le \mu_r\omega_d(t)\mathbb E\left[\bigl(\psi_t(U_t+1)-\psi_t(U_t)\bigr)^2\right];
\]
see \cite[Thm.~18.7]{LastPenrose2017}. For every \(k\in\mathbb N_0\),
\[
\left|\psi_t(k+1)-\psi_t(k)\right|
\le \mathbb E\left[\left|p(k+1+I_t+Z_t)-p(k+I_t+Z_t)\right|\right]
\le L_p.
\]
Therefore,
\[
\operatorname{Var}(\psi_t(U_t))
\le
\mu_r\omega_d(t)L_p^2.
\]
Since the variance term in \eqref{eq:moderate-cov-decomp} is
nonnegative, taking absolute values there and applying the preceding
estimate proves \eqref{eq:pointwise-gminus1}.
\end{proof}

Recall that \(S_{d-1}=dv_d\), and put
\[
J_d := \int_0^2t^{d-1}\omega_d(t)\, dt
 = M_d^{(\le1)}+M_d^{(>1)}.
\]
Since \(0\le\omega_d(t)\le1\),
\(
0\le J_d\le\frac{2^d}{d}.
\)

Integrating the preceding pointwise estimate provides a non-asymptotic bound for the short-range pair-correlation deviation. Substitution into the Stein estimate then yields a bound valid without any smallness assumption on \(\mu_r\).
\begin{theorem}[Non-asymptotic \(d_2\) bound]
\label{thm:moderate-r}
Let \(W\subset\R^d\) be a compact Borel set, and assume that
\(\lambda_r'>0\). Then
\begin{equation}
\label{eq:moderate-d2}
d_2\left(\mathcal L(Y_{r,W}),\mathcal L(\Pi_{r,W})\right)
\le |W|\Bigg[\lambda^2S_{d-1}r^d\left\{\mu_rL_p^2J_d
+\frac1d \left|m_p^+(\mu_r)^2-m_p(\mu_r)^2\right|\right\}+2\lambda_r'^2V_r
\Bigg].
\end{equation}

\end{theorem}

\begin{proof}
By polar coordinates and
Lemma~\ref{lem:pointwise-gminus1},
\[
\begin{aligned}
&\int_{\|h\|\le2r}
\left|g_{Y_r}^{\circ}(\|h\|)-1\right|\, dh
 = S_{d-1}r^d\int_0^2t^{d-1}\left|g_{Y_r}^{\circ}(tr)-1\right|\, dt\\
&\quad\le S_{d-1}r^d \Bigg\{\frac{\mu_rL_p^2}{m_p(\mu_r)^2} \int_0^2t^{d-1}\omega_d(t)\, dt
+ \left|\frac{m_p^+(\mu_r)^2}{m_p(\mu_r)^2}-1\right|\int_0^1t^{d-1}\, dt\Bigg\}\\
&\quad=S_{d-1}r^d\left\{\frac{\mu_rL_p^2}{m_p(\mu_r)^2}J_d
+\frac1d\left|\frac{m_p^+(\mu_r)^2}{m_p(\mu_r)^2}-1\right|\right\}.
\end{aligned}
\]
Substitution into Theorem~\ref{thm:stein-main} gives
\[d_2\left(\mathcal L(Y_{r,W}),\mathcal L(\Pi_{r,W})\right)
\le \lambda_r'^2|W|\Bigg[S_{d-1}r^d \left\{\frac{\mu_rL_p^2}{m_p(\mu_r)^2}J_d
+\frac1d \left|\frac{m_p^+(\mu_r)^2}{m_p(\mu_r)^2}-1 \right|\right\}
+ 2V_r\Bigg]. \]
Since
\( \lambda_r'=\lambda m_p(\mu_r),
\)
multiplying out the factors \(m_p(\mu_r)^{-2}\) proves
\eqref{eq:moderate-d2}.

\end{proof}

The preceding estimate can be simplified by controlling the difference between the shifted and unshifted retention means directly by \(L_p\). This removes the quantities \(m_p^+(\mu_r)\) and \(m_p(\mu_r)\) from the correlation term.
\begin{corollary}[Simplified non-asymptotic \(d_2\) bound]
Let \(W\subset\R^d\) be a compact Borel set, and assume that
\(\lambda_r'>0\). Then  
\begin{equation}
\label{eq:moderate-d2-simple}
d_2\left(\mathcal L(Y_{r,W}),\mathcal L(\Pi_{r,W})\right)
\le |W|\Bigg[\lambda^2S_{d-1}r^d\left\{\mu_rL_p^2J_d+\frac{2L_p}{d}\right\}
+2\lambda_r'^2V_r
\Bigg].
\end{equation}
\end{corollary}
\begin{proof}
    From the definition, 
\[
m_p^+(\mu_r)-m_p(\mu_r)
=\mathbb E\bigl[p(N_r+1)-p(N_r)\bigr],
\]
and hence
\[
\left|m_p^+(\mu_r)-m_p(\mu_r)\right|\le L_p.
\]
Because \(m_p^+(\mu_r),m_p(\mu_r)\in[0,1]\),
\[
\left|m_p^+(\mu_r)^2-m_p(\mu_r)^2\right|
=\left| m_p^+(\mu_r)-m_p(\mu_r) \right| \left(m_p^+(\mu_r)+m_p(\mu_r)\right)
\le 2 L_p.
\]
Applying this estimate to \eqref{eq:moderate-d2} proves
\eqref{eq:moderate-d2-simple}.
\end{proof}

\begin{remark}
The bounds require no smallness assumption on \(\mu_r\). They may,
however, become non-informative when \(r\) or \(\lambda\) is large.
Since \(d_2\le1\), the right-hand sides of
\eqref{eq:moderate-d2} and \eqref{eq:moderate-d2-simple} may always be
replaced by their minimum with \(1\).
\end{remark}

\subsection{Comparison of the bounds}
\label{subsec:compare-bounds}

The preceding results provide complementary forms of Poisson
approximation. They measure different types of error and emphasise
different properties of the retention rule \(p\). The direct coupling
controls total variation, the Stein estimates control the
Barbour--Brown \(d_2\)-metric, and
Theorem~\ref{thm:laplace-control} controls logarithms of Laplace
functionals. Their bounds should therefore not be compared only by
their numerical order.

\paragraph{Direct coupling in total variation.}
Theorem~\ref{thm:coupling-TV} gives
\begin{equation}
\label{eq:comparison-coupling}
d_{\mathrm{TV}}
\left(\mathcal L(Y_{r,W}),\mathcal L(\Pi_{r,W})\right)
\le \lambda|W|\delta_r,\qquad \delta_r:=\mathbb E \left|p(N_r)-m_p(\mu_r)\right|.
\end{equation}
This estimate is valid without any smallness assumption on \(\mu_r\).
Corollary~\ref{cor:small-mu-TV} gives
\(\delta_r\le2L_p\mu_r\), and hence
\[
d_{\mathrm{TV}}\left(\mathcal L(Y_{r,W}),\mathcal L(\Pi_{r,W})\right)
\le 2L_p\lambda|W|\mu_r.
\]

The coupling estimate is particularly effective when \(p\) is nearly
constant on the neighbour counts carrying most of the Poisson mass.
For example, if \(p(0)=p(1)=p(2)\), then, as \(\mu_r\downarrow0\),
\[
\delta_r
=O\bigl(\mathbb P\{N_r\ge3\}\bigr)
=O(\mu_r^3),
\]
and therefore
\[
d_{\mathrm{TV}}
\left(\mathcal L(Y_{r,W}),\mathcal L(\Pi_{r,W})\right)
=O\left(|W|\lambda\mu_r^3\right)
=O\left(|W|\lambda^4r^{3d}\right).
\]
If \(p\) is constant on all of \(\mathbb N_0\), then \(Y_r\) is an
independent thinning of \(X\), and the two laws coincide.

\paragraph{Stein bound in the small-radius regime.}
For \(p(0),p(1)>0\), Theorem~\ref{prop:stein-smallr}\emph{(a)} gives
\[
d_2\left( \mathcal L(Y_{r,W}), \mathcal L(\Pi_{r,W}) \right)
=O(r^d)
\]
as \(r\downarrow0\), with \(\lambda\) and \(W\) fixed. If
\(p(1)\ne p(0)\), Theorem~\ref{thm:fixed-lambda-small-r-sharpness}
shows that this order is sharp: there exist \(c,C>0\) such that
\[
cr^d \le d_2\left(\mathcal L(Y_{r,W}),\mathcal L(\Pi_{r,W})\right)
\le Cr^d
\]
for all sufficiently small \(r\).

If \(p(1)=p(0)>0\), the leading contact term in the
pair-correlation integral vanishes, and that integral contributes only
\(O(\mu_rr^d)=O(r^{2d})\). The neighbourhood-volume term
\(2V_r=2v_d(2r)^d\) in Theorem~\ref{thm:stein-main}, however, remains
of order \(r^d\). Thus, the present Stein estimate still gives only an
\(O(r^d)\) upper bound. The cancellation \(p(1)=p(0)\) improves the
pair-correlation part, but not the full bound.

When \(p(0)=0<p(1)\), Theorem~\ref{prop:stein-smallr}\emph{(b)} gives
\[
d_2\left(\mathcal L(Y_{r,W}),\mathcal L(\Pi_{r,W})\right)
=O\left(|W|\lambda^2r^d\right).
\]
In this case, \(g_{Y_r}^{\circ}\) grows as \(\mu_r^{-2}\) on the
mutual-inclusion region, but this growth is compensated by
\(\lambda_r'^2=O(\lambda^2\mu_r^2)\).

\paragraph{Non-asymptotic Stein bound.}
Without assuming that \(\mu_r\) is small,
Theorem~\ref{thm:moderate-r} gives
\begin{equation}
\label{eq:comparison-moderate}
d_2\left(\mathcal L(Y_{r,W}),\mathcal L(\Pi_{r,W})\right)
\le|W|\Bigg[\lambda^2S_{d-1}r^d\left\{\mu_rL_p^2J_d
+\frac1d \left|m_p^+(\mu_r)^2-m_p(\mu_r)^2\right|\right\}
+2\lambda_r'^2V_r\Bigg].
\end{equation}
Since
\(\left|m_p^+(\mu_r)^2-m_p(\mu_r)^2\right|\le2L_p\), the corresponding simpler bound is
\[
d_2\left(\mathcal L(Y_{r,W}),\mathcal L(\Pi_{r,W})\right)
\le
|W|\Bigg[\lambda^2S_{d-1}r^d\left\{\mu_rL_p^2J_d+\frac{2L_p}{d}\right\}
+2\lambda_r'^2V_r\Bigg].
\]

This estimate is most informative when \(L_p\) is small. For example, for the logistic family
\[
p_\beta(n)
=\frac{1}{1+\exp(\beta(n-n_0))},
\]
the mean-value theorem gives \(L_{p_\beta}\le\beta/4\). Hence the terms
in \eqref{eq:comparison-moderate} that measure the variation of
\(p_\beta\) decrease as \(\beta\downarrow0\). The residual term
\(2\lambda_r'^2V_r\) does not depend on \(L_p\), so this Stein bound
need not vanish as \(p\) approaches a constant rule. In that regime,
the direct coupling bound is more sensitive.

Since \(d_2\le1\), the right-hand sides of these estimates may always
be replaced by their minimum with \(1\).

\paragraph{Log--Laplace-functional control.}
Under the condition \(\lambda r^d\le c_d\),
Theorem~\ref{thm:laplace-control} gives, for every bounded measurable
\(g\ge0\) supported in \(W\),
\[
\left|\log L_{\Pi_r}(g)-\log L_{Y_r}(g)\right|\le C_d\lambda^2\|g\|_\infty|W|(2r)^d.
\]
This estimate does not involve a regularity modulus of \(p\). It
therefore applies to discontinuous or oscillatory retention rules,
such as \(p(n)=\mathbf 1\{n\text{ is even}\}\), provided that
\(\lambda r^d\le c_d\). Its conclusion concerns the logarithms of
Laplace functionals rather than a metric between the full
point-process laws.

\paragraph{Lower bounds and sharpness.}
Theorem~\ref{prop:d2-lower} gives
\[
d_2\left(\mathcal L(Y_{r,W}),\mathcal L(\Pi_{r,W})\right)
\ge\frac{\Delta_{r,W}^2}{4M_{r,W}},
\]
where
\[
\Delta_{r,W}
=\lambda_r'^2\int_{\|h\|\le2r}\left(g_{Y_r}^{\circ}(\|h\|)-1\right)|W\cap(W-h)|\, dh.
\]
This estimate depends on the signed, window-weighted integral of
\(g_{Y_r}^{\circ}-1\). It does not yield a lower bound in terms of
\[
\int_{\|h\|\le2r}\left|g_{Y_r}^{\circ}(\|h\|)-1\right|\, dh,
\]
because the positive and negative parts of
\(g_{Y_r}^{\circ}-1\) may cancel.

In the fixed-intensity regime
\(p(0)>0\), \(p(1)>0\), and \(p(1)\ne p(0)\),
Theorem~\ref{thm:fixed-lambda-small-r-sharpness} gives constants
\(c,C>0\) such that
\[
cr^d\le d_2\left(\mathcal L(Y_{r,W}),\mathcal L(\Pi_{r,W})\right)\le Cr^d
\]
for all sufficiently small \(r\). Thus the small-radius Stein upper
bound has the correct order in this case.

In summary, no single estimate is best in every regime:
\begin{itemize}
\item the direct coupling is strongest when \(p\) is nearly constant
on the likely neighbour counts and gives total-variation control;

\item the small-radius Stein estimate describes the effect of
short-range correlations and has the correct order when
\(p(1)\ne p(0)\);

\item the non-asymptotic Stein estimate gives an explicit bound in
terms of \(L_p\), but retains a geometric neighbourhood term;

\item the log--Laplace estimate does not require regularity of \(p\),
but assumes that \(\lambda r^d\) is sufficiently small and controls
only Laplace-functional differences.
\end{itemize}

\section*{Acknowledgements}
This work was supported by the National Natural Science Foundation of China under Grant No.~12201241, ``Stochastic Flow of Brownian Particles with Coalescence.''

The authors sincerely thank the anonymous reviewers for their careful reading of the manuscript and for their constructive comments and suggestions, which helped to improve both the presentation and the mathematical clarity of the paper.

% \bibliographystyle{cas-model2-names}
% \bibliography{references}

\end{document}